\documentclass{amsart}

\usepackage{mathrsfs}

\usepackage{amsfonts,amssymb,amsmath,amsthm}

\usepackage{color}

\usepackage{enumerate}

\def\NN{{\mathbb N}}

\def\ZZ{{\mathbb Z}}

\def\RR{{\mathbb R}}

\def\QQ{{\mathbb Q}}

\def\CC{{\mathbb C}}

\def\rho{\varrho}

\def\phi{\varphi}

\def\e {\mathrm{e}}

\newcommand{\ep}{\varepsilon}

\newcommand{\be}[1]{\begin{equation}\label{#1}}

\newcommand{\ee}{\end{equation}}

\newcommand{\multsum}[2]{\sum_{{\scriptstyle #1}\atop {\scriptstyle #2}}}

\def\dd{{\mathrm d}}

\def\NN{{\mathbb N}}

\def\QQ{{\mathbb Q}}\def\CC{{\mathbb C}}\def\ZZ{{\mathbb Z}}

\def \ds{\displaystyle}

\newtheorem{thm}{Theorem}

\newtheorem{prop}{Proposition}
\newtheorem{lem}{Lemma}[section]
\newtheorem{rem}{Remark}[section]

\makeatletter

\@addtoreset{equation}{section}
\makeatother

\makeatletter
\def\blfootnote{\xdef\@thefnmark{}\@footnotetext}
\makeatother

\title[ Rational points on an intersection of diagonal forms]
{Rational points on an intersection of diagonal forms} 

\author{Simon Boyer and Olivier Robert}

\begin{document}

\begin{abstract}We consider intersections of $n$ diagonal forms of degrees $k_1<\dots<k_n$, and we prove an asymptotic formula for the number of rational points of bounded height on these varieties. The proof uses the Hardy-Littlewood method and recent breakthroughs on the Vinogradov system. We also give a sharper result for one specific
value of $(k_1,\dots,k_n)$, using a technique due to Wooley and an estimate on exponential sums derived from a recent approach in  the van der Corput's method.

\end{abstract}

\maketitle


\section{Introduction}

Let $s,n\ge 1$ be integers. Let     $\mathbf{k}=(k_1,k_2,\dots,k_n)\in \NN^n$ such that
\be{kbold}
1\le k_1<k_2<\dots <k_n.
\ee
Let $\mathbf{F}:\RR^s\to \RR^n$ where $\mathbf{F}=(F_1,F_2,\dots,F_n)$ and 
$F_1,F_2,\dots,F_n\in \ZZ[t_1,t_2,\dots,t_s]$  are diagonal forms that satisfy
\be{defFi}
F_i(\mathbf{t})=\sum_{j=1}^su_{i,j}t_j^{k_i}, \quad  u_{i,j}\in \ZZ\smallsetminus \{0\}\qquad (1\le i\le n,\thinspace 1\le j\le s).
\ee

 We are interested in the asymptotic behaviour of the number of solutions of the following diophantine system
\be{syst1}
F_1(\mathbf{x})=F_2(\mathbf{x})=\dots=F_n(\mathbf{x})=0
\ee
with $\mathbf{x}\in \ZZ^s\cap [-X,X]^s$,    as $X\to +\infty$. In the sequel, this number will be written
 \be{NF}
N_{\mathbf{F}}(X):=\{\mathbf{x}\in \ZZ^s\cap [-X,X]^s\colon \mathbf{F}(\mathbf{x})=\mathbf{0}\}\qquad (X\ge 1),
\ee
and  for $\mathbf{k}$ as in \eqref{kbold} and $s\ge 1$, we set
\be{Dks}
\mathcal{D}(\mathbf{k},s)=\Big\{   \mathbf{F}=(F_1,F_2,\dots,F_n)\mbox{ that satisfy }  \eqref{defFi}    \Big\}.
\ee
The interest in  particular cases of systems of the form \eqref{syst1} has widely increased in the last decades, considering in particular for $\mathbf{k}$ as in \eqref{kbold} some  variations of  the historic case of the Vinogradov system, namely systems of the form
\be{vino-ki}
\sum_{j=1}^b \big(x_j^{k_i}-x_{b+j}^{k_i}\big)\qquad (1\le i\le n).
\ee
The original Vinogradov system itself,  the case where $k_i=i$  has been the subject of extensive studies  with its culminating point in  the last decade   with the    Vinogradov Mean Value Theorem,   through the efficient congruencing techniques due to Wooley (\cite{W1} and \cite{W2})  and the decoupling techniques due to Bourgain, Demeter and Guth \cite{BGD} (See also \cite{LP} for a remarkable survey on the Vinogradov system and these two methods). Namely, writing
\be{int-vino}
J_{b,k}(X):=
\int_{[0,1]^k}\Big|\sum_{1\le x\le X}e\Big(\sum_{1\le j\le k}\beta_j x^j\Big)\Big|^{2b}\dd \boldsymbol{\beta}\qquad (b, k\ge 1),
\ee
the  Vinogradov Mean Value Theorem asserts that for any fixed $\ep>0$ and any $b\ge 1$ on has
\be{main-vino}
J_{b,k}(X)\ll_{b,\ep}X^{\ep}\big(X^b+X^{2b-\frac{k(k+1)}{2}}\big)\qquad (X\ge 1).
\ee
Long before this Theorem  has been  proved,  it was  common knowledge in the fields of the Circle Method that whenever \eqref{main-vino} is verified for some $b\ge k(k+1)/2$, then one may derive an asymptotic for $J_{b+1,k}(X)$ as $X\to\infty$ (See for example \cite{vaughan}, \cite{W1}).
As a consequence,    the Vinogradov Mean value Theorem now implies an asymptotic for $J_{b,k}(X)$ as soon as $b\ge 1+k(k+1)/2$ (See \S 3.4    of     \cite{LP} and  see \cite{W2}).

\bigskip
 In \cite{W1}, Wooley   also states that the result extends  to \eqref{syst1} when  $\mathbf{F}$  satisfies \eqref{defFi},      in the particular case  $k_i=i$ $(1\le i\le n)$. Namely,  if  $s\ge 2n(n+1)+1$,   there exists a constant $c>0$ such that
\be{woo-asymp}
N_{\mathbf{F}}(X)\sim c X^{s-n(n+1)/2}\qquad (X\to\infty),
\ee
 provided that the system \eqref{syst1} has a nonsingular solution over $\RR$ and over all the $p$-adic $\QQ_p$.  One should point out  that the condition $s\ge 2n(n+1)+1$ above  is not a limitation of the Hardy-Littlewood method : it merely corresponds to the value $b\ge k(k+1)$ for which \eqref{main-vino} was known at the time of \cite{W1}. Since then,  \cite{W2} provides    an updated version  from the Vinogradov Mean value Theorem, with the new condition $s\ge n(n+1)+1$ for \eqref{woo-asymp}.

\bigskip
The aim of our paper is to derive an asymptotic for $N_{\mathbf{F}}(X)$ for more  general $\mathbf{k}$ as in \eqref{kbold}, and for  $\mathbf{F}\in \mathcal{D}(\mathbf{k},s)$ (with the notation \eqref{Dks}), when $s$ is sufficiently large, still provided that  the system \eqref{syst1} has 
a nonsingular solution over $\RR$ and over all the $p$-adic $\QQ_p$. Following the lines of \cite{vaughan} and \cite{W1}, we  use the Hardy-Littlewood method.
Namely, the classical  starting point is the identity
\be{NF-int}
N_{\mathbf{F}}(X)=\int_{[0,1]^n}\left(\sum_{\mathbf{x}\in I_s(X)}e\big(\boldsymbol{\alpha}\cdot \mathbf{F}(\mathbf{x})\big)\right)\dd \boldsymbol{\alpha},
\ee
 where  $I_s(X)= \ZZ^s\cap[-X,X]^s$, and   where here and in the sequel, $\boldsymbol{\alpha}\cdot\boldsymbol{\beta}$ denotes the usual scalar product  over $\RR^n$. 
For $\mathbf{k}\in \NN^n$ fixed as in   \eqref{kbold}, the number of solutions of the system \eqref{vino-ki} that satisfy $|x_j|\le X$ for each $j$ is is equal to
$\ds{\int_{[0,1]^n}\left| f_{\mathbf{k}}(\boldsymbol{\alpha};X) \right|^{2b}\dd \boldsymbol{\alpha}}$
where we have set
\be{fk}
f_{\mathbf{k}}(\boldsymbol{\alpha};X):=\sum_{|x|\le X}e\Big(\sum_{i=1}^n\alpha_i  x^{k_i}\Big).
\ee
Generalising the  heuristic argument for the Vinogradov system, it is conjectured that for any $\ep>0$ one has 
\be{heuristic-gene}
\int_{[0,1]^n}\left| f_{\mathbf{k}}(\boldsymbol{\alpha};X) \right|^{2b}\dd \boldsymbol{\alpha}\ll_{\ep}  X^{\ep}\big(X^b+X^{2b-  \sigma(\mathbf{k}) }\big)\qquad (X\ge 1)
\ee
where
\be{som-k}
\sigma(\mathbf{k})=\sum_{i=1}^nk_i.
\ee
In the current state of knowledge, this conjecture is verified for large and small values of $b$. More precisely, for large values, 
Fourier orthogonality yields the classical bound
\be{bound-full}
\int_{[0,1]^n}\left|    f_{\mathbf{k}}(\boldsymbol{\alpha};X)   \right|^{2b}\dd \boldsymbol{\alpha}\ll X^{\frac{k_n(k_n+1)}{2}- \sigma(\mathbf{k})}
J_{b,k_n}(2X+1)
\ee
and \eqref{main-vino} implies that \eqref{heuristic-gene} is satisfied for $b\ge k_n(k_n+1)/2$. In another direction, Corollary 1.2  of \cite{W2} implies that \eqref{heuristic-gene} is satisfied for $b\le n(n+1)/2$, which corresponds to the so-called  quasidiagonal behaviour.

\bigskip
We introduce two more classical objects  from  the Circle Method,   with a direct link to \eqref{NF-int} :   the singular integral
\be{int-sing}
\mathfrak{I}(\mathbf{F}):=\int_{\RR^n}\left(\int_{[-1,1]^s}e\big(\boldsymbol{\beta}\cdot \mathbf{F}(\mathbf{t})\big)\dd \mathbf{t}\right)\dd \boldsymbol{\beta}
\ee
 which measures the real density of the solutions  of  \eqref{syst1} in the box $[-1,1]^s$, and the singular series
\be{serie-sing}
\mathfrak{S}(\mathbf{F}):=\sum_{q\ge 1}\frac{1}{q^s}\sum_{\mathbf{a}\in A_n(q)}\sum_{\mathbf{r}\in [1,q]^s}e\left( \frac{ \mathbf{a}\cdot \mathbf{F}(\mathbf{r})  }{q}   \right),
\ee
with
\be{Anq}
A_n(q)=\big\{\mathbf{a}\in [1,q]^n \colon (a_1;a_2;\dots;a_n;q)=1  \big\},
\ee
related to the $p$-adic densities  of the solutions  of  \eqref{syst1}.  In the particular case of \eqref{woo-asymp}, the constant $c$ is $\mathfrak{I}(\mathbf{F})\mathfrak{S}(\mathbf{F})$, and the hypothesis about nonsingular solutions over $\RR$ and over the $p$-adic implies $c>0$.

\bigskip
We are now ready to state our first result, an analogue of \eqref{woo-asymp} for more general values of $\mathbf{k}$.

\begin{thm}\label{circle2}    Let $n\ge 2$ and $\mathbf{k}$ as in \eqref{kbold}. Let $s\ge 1+k_n(1+k_n)$  and  $\mathbf{F}\in \mathcal{D}(\mathbf{k},s)$. Then with the notation \eqref{int-sing} and \eqref{serie-sing}, both $\mathfrak{I}(\mathbf{F})$ and  $\mathfrak{S}(\mathbf{F})$  are convergent, and  for any $\ep>0$ one has 
\[
N_{\mathbf{F}}(X)=\mathfrak{I}(\mathbf{F})\mathfrak{S}(\mathbf{F})X^{s-\sigma(\mathbf{k})}+O(   X^{s-\sigma(\mathbf{k})-\eta_0+\ep})\qquad (X\ge 1)
\]
where $\sigma(\mathbf{k})$ has been defined in \eqref{som-k}, and where we have set $\eta_0=\frac{1}{nk_n^2}$.
If moreover the system \eqref{syst1} has a nonsingular solution over $\RR$ and over $\QQ_p$ for all $p$ , then $\mathfrak{I}(\mathbf{F})\mathfrak{S}(\mathbf{F})>0$.

\end{thm}

This result calls for several comments. First, Theorem 1 implies that the system \eqref{syst1} satisfies the Hasse Principle. One should also mention that  the constraint  $s\ge 1+ k_n(1+k_n)$ is directly related to \eqref{main-vino} with $b=k(k+1)/2$,  $k=k_n$   and its impact on \eqref{bound-full}. Furthermore, in the special case $k_i=i$ ($1\le i\le n$), we indeed recover Wooley's result \eqref{woo-asymp}.  We now comment on some other cases :  when $\mathbf{k}=(k,1)$,   Theorem \ref{circle2} improves  Theorem 1 of \cite{BR-2015}  for $k\ge 5$.
However, our result does not improve the cases $k=3$ and $k=4$, for which   the sharpest current results in the line of our theorem      remain   Theorem 1.5 of \cite{W3}  for $k=3$, and   Theorem 1 of \cite{BR-2015} for $k=4$.

\bigskip
Next, we  focus one particular case $n=3$ and $(k_1,k_2,k_3)=(1,3,5)$.  The corresponding Vinogradov-type system \eqref{vino-ki}  has already  been considered, in the frame of  paucity results (cf \cite{BR-2011}).  Theorem  \ref{circle2} applied to  $\mathbf{k}=(1,3,5)$ yields  the asymptotic as soon as  $s\ge 31$.
In the following result, we use a different approach  to  show that  in the case $\mathbf{k}=(1,3,5)$, we still have an asymptotic for $s=30$.

\begin{thm}\label{circle3}
Let $s=30$, $\mathbf{k}=(1,3,5)$ and   $\mathbf{F}\in \mathcal{D}(\mathbf{k},s)$.
Then with the notation \eqref{int-sing} and \eqref{serie-sing}, both $\mathfrak{I}(\mathbf{F})$ and  $\mathfrak{S}(\mathbf{F})$  are convergent, and  for any fixed $\ep>0$, one has 
\[
N_{\mathbf{F}}(X)=\mathfrak{I}(\mathbf{F})\mathfrak{S}(\mathbf{F})X^{21}+O(   X^{21-\frac{1}{8}+\ep})\qquad (X\ge 1).
\]
If moreover the system \eqref{syst1} has a nonsingular solution over $\RR$ and over $\QQ_p$ for all $p$ , then
$\mathfrak{I}(\mathbf{F})\mathfrak{S}(\mathbf{F})>0$.
\end{thm} 

The base of the proof of Theorem \ref{circle3} is still the Circle Method,  and our treatment of the major arcs is identical to that of Theorem \ref{circle2}. The main distinction
comes from the approach of the minor arcs, and makes a crucial use of  the structure of $(k_1,\dots,k_n)=(1,3,5)$, namely the  gap $k_n-k_{n-1}\ge 2$ between the two highest degrees.
More precisely, writing $f(\alpha_1,\alpha_2,\alpha_3)$ for the sum in \eqref{fk} for $\mathbf{k}=(1,3,5)$, and $\mathfrak{m}\subset [0,1]^3$ for the minor arcs, our aim is to obtain an upper bound of the form
\[
\int_{\mathfrak{m}}|f(\boldsymbol{\alpha})|^{30}\dd\boldsymbol{\alpha}\ll X^{21-\delta_0}\qquad (X\ge 1).
\]
Our proof proceeds essentially as follows :  we construct two suitable sets $\mathfrak{W}_2,\mathfrak{W_3}\subset [0,1]$ that resemble unions of  one-dimensional major arcs.

\medskip
 The first step is to bound the contribution of the $\boldsymbol{\alpha}\in \mathfrak{m}$ such that $\alpha_3\in \mathfrak{w_3}:=[0,1]\smallsetminus \mathfrak{W}_3$. Using a technique due to Wooley, for which the condition $k_n-k_{n-1}\ge 2$ is essential,      the integral of $|f|^{30}$ over $[0,1]^2\times \mathfrak{w_3}$  gives an admissible upper bound.

\medskip
For the next step, which is the main novelty in this paper, we give a more detailed sketch of the argument : writing $\mathfrak{w_2}:=[0,1]\smallsetminus \mathfrak{W}_2$, our aim is to bound
the contribution of the $\boldsymbol{\alpha}\in [0,1]\times \mathfrak{w_2}\times \mathfrak{W_3}$. For any interval $[z-\eta,z+\eta]$ counted in $\mathfrak{W}_3$, we have
\[
\int_{  [0,1]\times \mathfrak{w_2}\times [z-\eta,z+\eta]      }|f(\boldsymbol{\alpha})|^{30}\dd\boldsymbol{\alpha}\ll \sup_{\alpha_1,\alpha_3}\sup_{\alpha_2\in \mathfrak{w}_2}
|f(\boldsymbol{\alpha})|^{10}\int_{  [0,1]^2 \times [z-\eta,z+\eta]     }|f(\boldsymbol{\alpha})|^{20}\dd\boldsymbol{\alpha}
\]
The classical minor arc technique,  updated by the Vinogradov Mean Value Theorem, gives a suitable bound for the supremum, namely  a  saving that compensates the forthcoming summation over all intervals $[z-\eta,z+\eta]$. Hence, for the right hand side  integral, it is now sufficient to  have a saving close to $X^{-9}$.  Using the Beurling-Selberg function, we have
\[
\int_{  [0,1]^2 \times [z-\eta,z+\eta]     }|f(\boldsymbol{\alpha})|^{20}\dd\boldsymbol{\alpha}\ll \int_{  [0,1]^2\times [-\eta,\eta]     }|f(\boldsymbol{\alpha})|^{20}\dd\boldsymbol{\alpha}.
\]
Next, we produce an upper bound of the form $|f(\boldsymbol{\alpha})|\ll \frac{X}{(1+|\alpha_3|X^5)^{1/10}}$, ($|\alpha_3|\le \eta$)      by using a new formulation of van der Corput's upper bounds for exponential sums introduced by the second author in a recent work. Again, the condition $k_n-k_{n-1}\ge 2$ is essential here. This yields 
\[
\int_{  [0,1]^2\times [-\eta,\eta]     }|f(\boldsymbol{\alpha})|^{20}\dd\boldsymbol{\alpha}\ll \int_{-\eta}^{\eta}\frac{X^{10}}{1+|\alpha_3|X^5}\left(\int_{[0,1]^2}|f(\boldsymbol{\alpha})|^{10}\dd\alpha_1\dd\alpha_2\right)\dd\alpha_3.
\]
Using again the Beurling-Selberg function, we remove the dependency in $\alpha_3$ in the inner integral, and we are reduced to bounding the tenth moment for the linear and cubic, for which Hua's classical result is sufficient. Combined with  a simple integration over $\alpha_3$ for the remaining term, this  gives the expected saving. 

\medskip
Finally, for the last step, as classical trick in the  Circle Method, we use some pruning techniques to fill the gap between the complementary set  and the actual minor arcs.

\bigskip
The structure of our paper is as follows. In section 3, we give an asymptotic for Weyl sums, for the classical  Weyl sum as well as for the multidimensional version, mainly  in view of the major arcs and a simple estimate for the minor arcs. However, the range for these estimates goes slightly beyond what is required for this paper, and may be of independent interest.
In section 4, we study variations of Vinogradov's integrals using a technique developed in \cite{W}, and also an estimate of exponential sums in the style of van der Corput's method, in view of the proof of Theorem \ref{circle3}.
In section 5, we study the singular integrals and the singular series that occur in Theorems  \ref{circle2} and \ref{circle3}, following essentially  Parsell's and Schmidt's approach.
Section 6 and 7 are devoted to  the contribution of major arcs in in both theorems, as well as  the simplest estimate on minor arcs.
In section 8, we establish Theorem \ref{circle2} by proving a more general  result that does not depend directly on the Vinogradov's Mean Value Theorem.
At last, Section 9 is devoted to a refined estimate on minor arcs leading to  the proof of Theorem \ref{circle3}.

\bigskip
\section*{Acknowledgements}
The authors are very grateful to J. Br\"udern and T. Wooley for helpful conversations at an early stage of this work and remarks while writing this manuscript. The authors also wish to  thank the referee for a thorough reading of the manuscript and valuable comments. During the production of this manuscript, the second author was supported by the joint FWF-ANR project Arithrand: FWF: I 4945-N and ANR-20-CE91-0006.

\section{Notation} 

For any integers $a_1,a_2,\dots, a_n$, we set $(a_1;a_2;\dots ;a_n)=\mathrm{gcd}(a_1,a_2,\dots,a_n)$. Similarly, whenever $\mathbf{a}\in \ZZ^n$ and $q\ge 1$,  we write $(\mathbf{a};q)$  for the gcd $(a_1;a_2;\dots ;a_n;q)$.
For any $k\in \NN$ and any $\boldsymbol{\alpha}=(\alpha_1,\dots,\alpha_k)$, $\boldsymbol{\beta}=(\beta_1,\dots,\beta_k)\in \RR^k$, the product $\boldsymbol{\alpha}\cdot\boldsymbol{\beta}$ denotes the usual scalar product, whereas $\boldsymbol{\alpha}\otimes \boldsymbol{\beta}$ denotes the tensor product $(\alpha_1\beta_1,\alpha_2\beta_2,\dots, \alpha_k\beta_k)$. Through the paper, the small  letter $p$ (with or without index) represents a prime  number.

\section{Exponential sums and oscillating integrals}

\subsection{A truncated Poisson Formula}
\begin{lem}[Lemma 2.3 of \cite{Brandes2020}]\label{poisson} Suppose that $\phi$ is  a twice continuously differentiable function on an interval $I$ and let $H>2$ be a number such that  $|\phi'(x)|\le H $ for all $x\in I$.  
Suppose  further that   $\phi''$ has at most finitely many zeros in the interval $I$.   Then 
\[
\sum_{n\in I}e\big(\phi(n)\big)=\sum_{|h|\le H}\int_Ie\big(\phi(t)-ht\big)\dd t+O(\log H).
\]
\end{lem}

In the sequel, we set
\be{Pkbeta}
P_k(\boldsymbol{\beta};t)=\sum_{i=1}^k\beta_it^i\qquad (k\ge 1,\thinspace  \boldsymbol{\beta}\in \RR^k,\thinspace t\in\RR),
\ee

\subsection{Estimates on complete sums}

\begin{lem}\label{som-complet} 
Let $k\ge 2$ and $\ep>0$ fixed. With the notation \eqref{Pkbeta}, For any $q\in \NN$ and any $\mathbf{a}=(a_1,\dots,a_k)\in \ZZ^k$, we have the following estimates :
\begin{enumerate}[(i)]
\item One has
\[
\sum_{r=1}^qe\Big(\frac{P_k(\mathbf{a};r)}{q}\Big)\ll_{\ep,k}(\mathbf{a};q)^{1/k}q^{1-\frac{1}{k}+\ep}
\]
\item If moreover $(\mathbf{a};q)=1$ and $H\gg q$, then
\[
\sum_{|h|\le H}\left|\sum_{r=1}^qe\Big(\frac{P_k(\mathbf{a};r)+hr}{q}\Big)   \right|\ll_{\ep,k}Hq^{1-\frac{1}{k}+\ep}
\]
and
\[
\sum_{|h|\le H}\frac{1}{h}\left|\sum_{r=1}^qe\Big(\frac{P_k(\mathbf{a};r)+hr}{q}\Big)   \right|\ll_{\ep,k}q^{1-\frac{1}{k}+\ep}\log(2+H)
\]
\item If $(\mathbf{a};q)=1$ and $\mathbf{w}\in \ZZ^k$ with $w_j\neq 0$ for each $j$, then
\[
\sum_{r=1}^qe\Big(\frac{P_k(\mathbf{a}\otimes \mathbf{w};r)}{q}\Big)\ll_{\ep,k}\Big(\prod_{j=1}^k|w_j|\Big)^{1/k}q^{1-\frac{1}{k}+\ep}
\]
\end{enumerate}
\end{lem}

\begin{proof} The bound (i) is essentially a reformulation of Theorem 7.1 of  \cite{vaughan}.
For  the first bound of  (ii), we use  (i) for the inner sum, and we now have to bound
$\ds{\sum_ {|h|\le H}(a_1+h;a_2;a_3;\dots;a_n)^{1/k}}$. For  a fixed $d\mid q$, the contribution in this sum  of the $h$ such that  $(a_1+h;a_2;a_3;\dots;a_n)=d$ is $\ll d^{1/k}(\frac{H}{d}+1)$ since, $d\mid h+a_1$. Summing over the $O(q^{\ep})$ choices for $d$ gives the expected result.  The proof for the second  bound of (ii) is quite similar, we omit the details.  
Finally, the bound (iii) is a consequence of (i), by noticing that $(\mathbf{a}\otimes \mathbf{w};q)$ divides $ |w_1w_2\dots w_k|$.
\end{proof}

\subsection{Estimates on integrals}

\begin{lem}\label{ipp}Let $c_0,C_0>0$ such that $c_0<1$. For any $X>0$, $A>0$ and any $C^2$ function $\varphi:\RR\to \RR$ such that
$|\varphi'(t)|\le c_0A$,  $|\varphi''(t)|\le C_0A/X$  $(-X\le t\le X)$,  one has
\[
\left|\int_{-X}^Xe\big(\varphi(t)-At\big)\dd t\right|\le  \frac{C_0+1-c_0}{(1-c_0)^2\pi A}.
\]
\end{lem}
\begin{proof}
Writing 
\[
\int_{-X}^Xe\big(\varphi(t)-At\big)\dd t=\int_{-X}^X\frac{\big(e\big(\varphi(t)-At\big)\big)'}{2i\pi(\varphi'(t)-A)}\dd t,
\]
this is merely an integration by parts, using the fact that  $|\varphi'(t)-A|\ge (1-c_0) A$ for $-X\le t\le X$.
\end{proof}

\begin{lem}\label{maj-int-osc}Let $k\ge 2$. With the notation \eqref{Pkbeta}, one has
\[
\int_{-1}^1e\big(P_k(\boldsymbol{\beta};t)\big)\dd t\ll_{k}\Big(1+\sum_{1\le j\le k}|\beta_j|\Big)^{-1/k}\qquad (\boldsymbol{\beta}\in \RR^k)
\]
\end{lem}
\begin{proof}
This is a direct consequence of Theorem 7.3  of \cite{vaughan}.
\end{proof}

\begin{lem}\label{maj-int-2}
We have the following estimates :
\begin{enumerate}[(i)]
\item For any $n\ge 1$, $\sigma\in \RR$,  $U>0$, we have
\[
\int_{[-U,U]^n}\Big(1+\sum_{i=1}^n|\beta_i|\Big)^{-\sigma}\dd \boldsymbol{\beta}\le \frac{2^n}{(n-1)!}\Big(1+(1+nU)^{n-\sigma} \Big)\log(1+U).
\]
\item We have
\[
\int_{\RR^n\smallsetminus [-U,U]^n} \Big(1+\sum_{i=1}^n|\beta_i|\Big)^{-\sigma}\dd \boldsymbol{\beta}\ll_{\sigma,n}U^{n-\sigma}\qquad (n\ge 1,\thinspace \sigma>n,\thinspace U\ge 1).
\]
\item We have 
\[
\int_{H_n(U)} \Big(1+\sum_{i=1}^n|\beta_i|\Big)^{-\sigma}\dd \boldsymbol{\beta}\ll_{\sigma,n}U^{n-\sigma}\qquad (n\ge 1,\thinspace \sigma>n,\thinspace U\ge 1).
\]
where we have set $\ds{H_n(U)=\Big\{ \boldsymbol{\beta}\in \RR^n\colon \sum_{i=1}^n|\beta_i|\ge U \Big\}}$.
\end{enumerate}
\end{lem}

\begin{proof} For (i), writing $I(\sigma,U)$ for  the integral on the left hand side, we start with the case $\sigma=n$, and a direct computation gives $I(n,U)\le \frac{2^n}{(n-1)!}\log(1+U)$.
For $\sigma<n$, we have $I(\sigma,U)\le (1+nU)^{n-\sigma}I(n,U)$ which implies the conclusion. Finally for $\sigma>n$, we have $I(\sigma,U)\le I(n,U)$ which yields  again the expected result.
For (ii), by symmetry, one may assume that $|\beta_n|>U$, and a direct computation gives the result. Finally, (iii) is a consequence of (ii).
\end{proof}

\subsection{Estimates on Weyl sums}

Our main result in this section is an estimate for the generating function in the Vinogradov system, crucial in the treatment of the major arcs. As usual, the main term involves the complete sum and the integral associated.

\begin{thm}\label{weyl-asymp}
With the previous notation, for any fixed $\ep>0$ and $k\ge 1$, one has
\[
\sum_{|x|\le X}e\big( P_k\big( \frac{\boldsymbol{a}}{q}+\boldsymbol{\beta};x  \big)    \big)=\frac{ X }{  q }
\Big( \sum_{r=1}^qe\Big(\frac{P_k(\boldsymbol{a};r)    }{q}   \Big)    \Big)
\int_{-1}^1e\big( P_k(\boldsymbol{\beta};Xt)  \big)\dd t +O\Big( \big( \frac{q}{(q;\mathbf{a})}  \big)^{1-\frac{1}{k}+\ep}  \Big)
\]
uniformly for $\boldsymbol{a}\in \ZZ^k$, $q\ge 1$ and $\boldsymbol{\beta}\in \RR^k$ such that
$ |\beta_1|\le \tfrac{1}{2q}$,     $\ds{\sum_{2\le j\le k}j|\beta_j|X^{j-1}\le \tfrac{1}{4q}}$.

\end{thm}
\begin{proof}
We start with the particular case $(\mathbf{a};q)=1$.  The first lines of our  proof follow the standard approach   used in  Theorem 3 of  \cite{BR-2015} and Theorem 4.1 of \cite{vaughan}.   One has 
\[
S:=\sum_{|x|\le X}e\big( P_k\big( \frac{\boldsymbol{a}}{q}+\boldsymbol{\beta} ;x \big)    \big)=
\sum_{r=1}^qe\Big(\frac{P_k(\boldsymbol{a};r)    }{q}   \Big)\multsum{|x|\le X}{x\equiv r \mod q}e\big(P_k(\boldsymbol{\beta};x)\big)
\]
\[
=\frac{1}{q}\sum_{r=1}^qe\Big(\frac{P_k(\boldsymbol{a};r)    }{q}   \Big)\sum_{|x|\le X}e\big(P_k(\boldsymbol{\beta};x)\big)\sum_{-q/2<b\le q/2}e\big(\frac{b(r-x)}{q}\big)
\]
\[
=\frac{1}{q}\sum_{-q/2<b\le q/2}\left(  \sum_{r=1}^qe\Big(\frac{P_k(\boldsymbol{a};r) +br   }{q}   \Big)  \right)
\left(    \sum_{|x|\le X}e\Big(P_k(\boldsymbol{\beta};x)-\frac{bx}{q}\Big)   \right)
\]
For each $b,q,\boldsymbol{\beta}$, the inner sum over $x$ meets the requirements of Lemma \ref{poisson} with $H=3$ so that
$S=S_1+O(S_2)$  with
\[
S_1=\frac{1}{q}\sum_{-q/2<b\le q/2}\left(  \sum_{r=1}^qe\Big(\frac{P_k(\boldsymbol{a};r) +br   }{q}   \Big)  \right)
  \Big(\sum_{|h|\le 3}\int_{-X}^X e\big( P_k(\boldsymbol{\beta};t)-\frac{bt}{q}-ht  \big)\dd t \Big)
\]
and
\[
S_2=\frac{1}{q}\sum_{-q/2<b\le q/2}\left|  \sum_{r=1}^qe\Big(\frac{P_k(\boldsymbol{a};r) +br   }{q}   \Big)  \right|.
\]
From (ii) of Lemma \ref{som-complet}, we have  $S_2\ll_{\ep} q^{1-\frac{1}{k}+\ep}$. Now, writing $m=qh+b$ with $|h|\le 3$ and $-\frac{q}{2}<b\le \frac{q}{2}$, one has
\[
S_1=\frac{1}{q}\sum_{-7q/2<m\le 7q/2}\left(  \sum_{r=1}^qe\Big(\frac{P_k(\boldsymbol{a};r) +mr   }{q}   \Big)  \right)
  \int_{-X}^X e\big( P_k(\boldsymbol{\beta};t)-\frac{mt}{q} \big)\dd t.
\]
At this point we differ from the argument used   in the Theorem 3 of  \cite{BR-2015} :  for each $m\neq 0$,    Lemma \ref{ipp} yields
\be{deriv1}
  \int_{-X}^X e\big( P_k(\boldsymbol{\beta};t)-\frac{mt}{q} \big)\dd t \ll \frac{q}{|m|}
\ee
so that
\[
S_1=   \frac{1}{q}\left(  \sum_{r=1}^qe\Big(\frac{P_k(\boldsymbol{a};r) }{q}   \Big)  \right)
  \int_{-X}^X e\big( P_k(\boldsymbol{\beta};t)\big)\dd t+O(S_3)
\]
with
\[
S_3=\sum_{1\le|m|\le 7q/2}\frac{1}{m}\left| \sum_{r=1}^qe\Big(\frac{P_k(\boldsymbol{a};r) +mr   }{q}   \Big)  \right|.
\]
Now, from (ii) of Lemma \ref{som-complet}, we have  $S_3\ll_{\ep} q^{1-\frac{1}{k}+\ep}$ : this completes the proof in the case $(\mathbf{a};q)=1$ after an obvious linear change of variable in the inner integral.

 \bigskip
For the general case, writing $q'=a/(\mathbf{a},q)$ and $\mathbf{a}'=\mathbf{a}/(\mathbf{a},q)$, we apply the previous estimate
with $q'$ and $\mathbf{a}'$, and we conclude by observing that
\[
\frac{1}{q'} \sum_{r=1}^{q'}e\Big(\frac{P_k(\boldsymbol{a}';r) }{q'}   \Big) =\frac{1}{q} \sum_{r=1}^qe\Big(\frac{P_k(\boldsymbol{a};r) }{q}   \Big) 
\]

\end{proof}

\begin{rem}  In the case of particular polynomial phases, results sharper than our Theorem \ref{weyl-asymp} may be obtained :  we have already mentioned  Theorem 4.1 of \cite{vaughan} and Theorem 3 of  \cite{BR-2015} for the case where the phase is either a monomial or a monomial and a linear term. In that case the range of validity for $\boldsymbol{\beta}$ is significantly wider. For such particular phase, the sharpest current result is  Theorem 1.1 of \cite{Brandes2020}, with a new main term and a sharper error term.  In the case of our Theorem 3, the polynomial phase is more general,  and in particular  $t\mapsto P_k(\boldsymbol{a};t)$ does not necessarily have a monotonic first derivative, which was a crucial aspect in Theorem 3 of  \cite{BR-2015}. Hence, the bound \eqref{deriv1} is not  a consequence of van der Corput's result for the first  derivative : instead, we use Lemma \ref{ipp}, which induces a constraint on
$\boldsymbol{\beta}$.
\end{rem}

In the sequel, our aim is to apply Theorem  \ref{weyl-asymp}  to the Weyl sums related to  \eqref{fk}, including a multidimensional version. Considering $\boldsymbol{\varphi}:\RR\to \RR^n$ defined by
\be{phi-curve}
\boldsymbol{\varphi}(t)=(t^{k_1},\dots,t^{k_n})\qquad (t\in \RR),
\ee
the sum in \eqref{fk} may be written
\[
f_{\mathbf{k}}(\boldsymbol{\alpha};X):=\sum_{|x|\le X}e\big( \boldsymbol{\alpha}\cdot \boldsymbol{\varphi}(x)\big)\qquad (\boldsymbol{\alpha}\in \RR^n).
\]
In order to introduce the multidimensional sums, we consider the corresponding complete sum
\be{sk}
S_{\mathbf{k}}(q;\mathbf{a})=\sum_{r=1}^qe\Big(\frac{ \boldsymbol{a}\cdot \boldsymbol{\varphi}(r)  }{q}\Big)
\ee
and the corresponding integral
\be{vk}
v_{\mathbf{k}}(\boldsymbol{\beta};X)=\int_{-1}^1e\big(  \boldsymbol{\beta}\cdot \boldsymbol{\varphi}(Xt)  \big)\dd t.
\ee

For a fixed  $\mathbf{F}\in \mathcal{D}(\mathbf{k},s)$, we now   derive a similar estimate for the associated multidimensional Weyl sum.
With $u_{i,j}$ defined by \eqref{defFi}, we set 
\be{ujgras}
\mathbf{u}_j=(u_{1,j},u_{2,j},\dots,u_{n,j})\qquad (1\le j\le s).
\ee
and
\be{Finfini}
\|F\|_{\infty}:=\max_{i,j}|u_{i,j}|.
\ee
The multidimensional analogues  now read 
\be{sum-weyl-multi}
 f[\mathbf{F}] (\boldsymbol{\alpha};X) :=     \sum_{\mathbf{x}\in I_s(X)}e\Big(\boldsymbol{\alpha}\cdot \mathbf{F}(\mathbf{x})\Big)=\prod_{j=1}^sf_{\mathbf{k}}(\mathbf{u}_j\otimes \boldsymbol{\alpha};X),
\ee
\be{sum-r-mult}
 S [\mathbf{F}] (q;\mathbf{a}):=     \sum_{\mathbf{r}\in [1,q]^n}e\Big(\frac{ \boldsymbol{a}\cdot \mathbf{F}(\mathbf{r})  }{q}\Big)=\prod_{j=1}^sS_{\mathbf{k}}(q;\mathbf{u}_j\otimes \boldsymbol{a}),
\ee
\be{int-beta-mult}
 v  [\mathbf{F}]   (\boldsymbol{\beta};X)   :=   \int_{[-1,1]^s}e\big(  \boldsymbol{\beta}\cdot \boldsymbol{F}(X\mathbf{t})  \big)\dd \mathbf{t} =\prod_{j=1}^sv_{\mathbf{k}}(\mathbf{u}_j\otimes \boldsymbol{\beta};X).
\ee

Moreover, for  $\mathfrak{q}\ge 1$ and $X\ge 1$, we introduce a condition on  $\boldsymbol{\gamma}\in \RR^n$ to have suitable coordinates as follows :
\be{cond-beta}
\begin{cases}
\displaystyle{ |\gamma_{1}|\le\frac{1}{2\mathfrak{q}}\mbox{ and }\sum_{i=2}^nk_i|\gamma_i|X^{k_i-1}\le \frac{1}{4\mathfrak{q}} }        &  \mbox{ if }k_1=1\\
\displaystyle{  \sum_{i=1}^nk_i|\gamma_i|X^{k_i-1}\le \frac{1}{4\mathfrak{q}}   }     &   \mbox{ if }k_1>1.    \\
\end{cases}
\ee
Finally, we define
\be{xi-k}
\xi_{\mathbf{k}}(\boldsymbol{\beta},X):=1+\sum_{i=1}^n|\beta_i|X^{k_i}\qquad (\mathbf{k}\in \NN^n,\thinspace \boldsymbol{\beta}\in \RR^n).
\ee
\begin{lem}\label{weyl-mult}
Let $s\ge 1$,  $\mathbf{k}$ as in \eqref{kbold}, and  $\mathbf{F}\in  \mathcal{D}(\mathbf{k},s)$. Then with the notation \eqref{sum-weyl-multi} to \eqref{xi-k}, uniformly for   $q\ge 1$, $\mathbf{a}\in A_n(q)$ and $\boldsymbol{\beta}\in \RR^n$ such that $\boldsymbol{\gamma}:=\|\mathbf{F}\|_{\infty} \boldsymbol{\beta}$ satisfies  \eqref{cond-beta} with $\mathfrak{q}=q$ , we have
\[
 f[\mathbf{F}] \Big(\frac{\mathbf{a}}{q}+\boldsymbol{\beta};X\Big)=\frac{X^s}{q^s} \big(S [\mathbf{F}] (q;\mathbf{a})\big)\big(v  [\mathbf{F}]   (\boldsymbol{\beta};X) \big)+O(E)
 \]
where we have set
$E=X^{s-1}q^{1-\frac{s}{k_n}+\ep}   \xi_{\mathbf{k}}(\boldsymbol{\beta},X)^{-(s-1)/k_n}+q^{s-\frac{s}{k_n}+\ep}$.
\end{lem}

\begin{proof} Our proof uses the following :  for any $s\ge 1$ and $z_1,z_2,\dots z_s,\delta_1,\dots,\delta_s\in \CC$ such that $|z_j|\le Z$ and $|\delta_j|\le \Delta$ for any $j$, one has
\be{diff-prod}
\Big|\prod_{j=1}^s(z_j+\delta_j)- \prod_{j=1}^sz_j\Big|\le s2^{s-2}\left(\Delta Z^{s-1}+\Delta^s\right).
\ee

We now use the notation introduced before our lemma. For any  choice $\mathbf{w} =\mathbf{u}_j$, we have $|w_i\beta_i|\le \|\mathbf{F}\|_{\infty}|\beta_i|$ for each $i$. Hence $\mathbf{w}\otimes\boldsymbol{\beta}$ satisfies  \eqref{cond-beta} and thus   meets the requirements  of   Theorem \ref{weyl-asymp}. Hence we have
\be{estim-fk}
f_{\mathbf{k}}(\mathbf{w}\otimes\boldsymbol{\alpha};X)=X\frac{ S_{\mathbf{k}}(q;\mathbf{w}\otimes\mathbf{a}) }{q}v_{\mathbf{k}}(\mathbf{w}\otimes\boldsymbol{\beta};X)+O_{\mathbf{k},\mathbf{w},\ep}\Big(q^{1-\frac{1}{k_n}+\ep}\Big).
\ee
From this estimate, (iii) of Lemma \ref{som-complet}  with  Lemma  \ref{maj-int-osc}  give the upper bound
\be{maj-fk}
f_{\mathbf{k}}(\mathbf{w}\otimes\boldsymbol{\alpha};X)\ll_{\mathbf{k},\mathbf{w},\ep}Xq^{\ep}
\Big(         q \xi_{\mathbf{k}}(\boldsymbol{\beta},X)\Big)^{-\frac{1}{k_n}}+   q^{1-\frac{1}{k_n}+\ep}.
\ee
Finally, using \eqref{diff-prod}, we deduce
\[
\prod_{j=1}^sf_{\mathbf{k}}\left(\mathbf{u}_j\otimes \big( \frac{\mathbf{a}}{q}+\boldsymbol{\beta}\big);X\right)-   \prod_{j=1}^s\left(\frac{X}{q}S_{\mathbf{k}}\big(q;\mathbf{u}_j\otimes \boldsymbol{a}\big)v_{\mathbf{k}}\big(\mathbf{u}_j\otimes \boldsymbol{\beta};X\big)\right)
\]
\[
\ll \frac{X^{s-1}}{q^{s-1}}\Big(q^{1-\frac{1}{k_n}+\ep}\Big)^s\xi_{\mathbf{k}}(\boldsymbol{\beta},X)^{-(s-1)/k_n}+    \Big(q^{1-\frac{1}{k_n}+\ep}\Big)^s
\]
which implies the expected result.
\end{proof}

\section{Vinogradov  integrals and generalisations}

\subsection{Some classical inequalities on moments of exponential sums}

The following results is an application of Fourier orthogonality that generalises the argument that gives \eqref{bound-full}.

\begin{lem}\label{ampli}Let $\Omega$ be a finite set, and $\Phi:\Omega\to \CC$. For any $m\in \NN$ and $\mathbf{G}:\Omega\to \ZZ^m$, consider the  set
$\mathcal{H}=\{ \mathbf{G}(\omega)-\mathbf{G}(\omega')\colon (\omega,\omega')\in \Omega^2\}\subset \ZZ^m$. Then
\[
\Big|\sum_{\omega\in \Omega}\Phi(\omega)\Big|^2\le (\#\mathcal{H})\int_{[0,1]^m}\Big|\sum_{\omega\in \Omega}\Phi(\omega)e\big( \boldsymbol{\beta}\cdot \mathbf{G}(\omega)\big)\Big|^2\dd\boldsymbol{\beta}.
\]
\end{lem}
\begin{proof}
We have
\[
\Big|\sum_{\omega\in \Omega}\Phi(\omega)\Big|^2=\sum_{(\omega,\omega')\in \Omega^2}\Phi(\omega)\overline{\Phi(\omega')}
=\sum_{\mathbf{h}\in \mathcal{H}}\multsum{(\omega,\omega')\in \Omega^2}{ \mathbf{G}(\omega)-\mathbf{G}(\omega')=\mathbf{h} }\Phi(\omega)\overline{\Phi(\omega')}
\]
Now, by Fourier orthogonality, for any $\mathbf{h}\in \mathcal{H}$, we have 
\[
A(\mathbf{h}):=\multsum{(\omega,\omega')\in \Omega^2}{ \mathbf{G}(\omega)-\mathbf{G}(\omega')=\mathbf{h} }\Phi(\omega)\overline{\Phi(\omega')}
=\int_{[0,1]^m}\Big|\sum_{\omega\in \Omega}\Phi(\omega)e\big( \boldsymbol{\beta}\cdot \mathbf{G}(\omega)\big)\Big|^2 \e(- \boldsymbol{\beta}\cdot \mathbf{h}  )     \dd\boldsymbol{\beta}
\]
which gives the expected result by using the bound $A(\mathbf{h})\le A(\mathbf{0})$.
\end{proof}

\begin{lem}\label{long-sum} Let $r\in \NN$  and  $\varphi:\ZZ\to \CC$. For any $a,b,c,d\in \ZZ$ such that $c\le a<b\le d$ one  has
\[
\Big|\sum_{a\le x\le b}\varphi(x)\Big|^{r}\le \big(1+\log(d-c+1)\big)^{r-1}\int_{-1/2}^{1/2}\Big|\sum_{c\le x\le d}\varphi(x)e(\gamma x)\Big|^{r}\min\big(d-c+1,\tfrac{1}{2|\gamma|}\big)\dd \gamma.
\]
\end{lem}
\begin{proof} This is essentially Lemma 2.1 of \cite{sargos} followed by H\"older's inequality.
\end{proof}

\subsection{A technique due to Wooley related to partial minor arcs}

Let $k\ge 3$ fixed. For any $X\ge 1$ and any $\theta\in \RR$, we set 
\be{}
\psi_k(\theta,\mu;X):= \frac{1}{X} \sum_{1\le y\le X}\min\Big(X^{k-1}, \frac{1}{\|k\theta y+\mu\|}\Big)
\ee

\be{sup-psi}
\psi_k^{*}(\theta;X):=\sup_{\mu\in \RR}  \psi_k(\theta,\mu;X)
\ee

\be{gk}
g_k(\boldsymbol{\alpha},\theta;X)=\sum_{|x|\le X}e\Big(P_{k-2}(\boldsymbol{\alpha}; x)+\theta x^k\Big) \qquad \big((\boldsymbol{\alpha},\theta)\in \RR^{k-2}\times \RR\big).
\ee
where $P_{k}(\boldsymbol{\alpha};x)$ has been defined in \eqref{Pkbeta}.

\bigskip
The next theorem is a reformulation and a slight generalisation of the  crucial argument in the proof of Theorem 2.1 of \cite{W}.

\begin{thm}\label{missing-deg}Let   $k\ge 3$ and $b\ge 1$ be fixed. Then, with the notation \eqref{Pkbeta}, \eqref{sup-psi}, \eqref{gk} and \eqref{int-vino}, one has
\[
\int_{[0,1]^{k-2}\times A}  \left|   g_k(\boldsymbol{\alpha},\theta;X)        \right|^{2b}  \dd\boldsymbol{\alpha}\dd\theta\ll_{b,k} (\log(4X))^{2b}   \Big(\sup_{\theta\in A}\psi_k^{*}(\theta;X) \Big)       J_{b,k}(4X+1)
\]
uniformly for $X\ge 1$, $c\in \RR$ and $A\subset [c,c+1]$, where $A$ is a Lebesgue-measurable set.
\end{thm}

\begin{proof} One has
\[
     g_k(\boldsymbol{\alpha},\theta;X) =\sum_{y-X\le x\le y+X}e\Big(P_{k-2}(\boldsymbol{\alpha}; x-y)+\theta (x-y)^k\Big)\qquad (1\le y\le X).
\]
Since $[y-X,y+X]\subset [-2X,2X]$, Lemma  \ref{long-sum} yields
\[
   |g_k(\boldsymbol{\alpha},\theta;X)  |^{2b}=              \left| \sum_{y-X\le x\le y+X}e\Big(P_{k-2}(\boldsymbol{\alpha}; x-y)+\theta (x-y)^k\Big) \right|^{2b}
\]
\[
\ll \big(\log (4X)\big)^{2b-1}\int_{-1/2}^{1/2}   \left|   \sum_{|x|\le 2X}e\Big(\gamma x + P_{k-2}(\boldsymbol{\alpha}; x-y)+\theta (x-y)^k\Big) \right|^{2b}        \min\big(4X,\tfrac{1}{|\gamma|}\big)\dd \gamma
\]
uniformly for $1\le y\le X$. Now, integrating over $\boldsymbol{\alpha}$, and averaging over $y$, we have
\be{g-theta}
G(\theta)
\ll  (\log(4X))^{2b-1}\frac{1}{X}\sum_{1\le y\le X}\int_{-1/2}^{1/2}   I(\gamma,y;\theta)              \min\big(4X,\tfrac{1}{|\gamma|}\big)\dd \gamma,
\ee
where we have set $\ds{ G(\theta)=\int_{[0,1]^{k-2}}  \left|       g_k(\boldsymbol{\alpha},\theta;X) \right|^{2b}  \dd\boldsymbol{\alpha}}$
and
\[
 I(\gamma,y;\theta):=\int_{[0,1]^{k-2}} \left|   \sum_{|x|\le 2X}e\Big(\gamma x + P_{k-2}(\boldsymbol{\alpha}; x-y)+\theta (x-y)^k\Big) \right|^{2b}     \dd\boldsymbol{\alpha}.
\]
Writing
\[
\sigma_j(\mathbf{x};y):=\sum_{m=1}^b\left( (x_m-y)^j-(x_{b+m}-y)^j \right),\quad \sigma_j(\mathbf{x}):=\sigma_j(\mathbf{x};0)\qquad (\mathbf {x}\in \ZZ^s),
\]
we have $\ds{I(\gamma,y;\theta)=\sum_{\mathbf{x}\in J_1(y)}e\big(\gamma \sigma_1(\mathbf{x}) +\theta\sigma_k(\mathbf{x};y)  \big)}$ 
where $J_1(y)$ is the set of   solutions of the system 
\[
\sigma_j(\mathbf{x};y)=0\qquad (1\le j\le k-2)\qquad \big(  \mathbf{x}\in [-2X,2X]^{2b}    \big).
\]
By translation invariance of $J_1(y)$, we have $J_1(y)=J_1(0)$, and $\gamma \sigma_1(\mathbf{x}) +\theta\sigma_k(\mathbf{x};y) =\theta\sigma_k(\mathbf{x}) -ky\theta\sigma_{k-1}(\mathbf{x})$ for $\mathbf{x}\in J_1(0)$. Hence
\[
I(\gamma,y;\theta)=\sum_{\mathbf{x}\in J_1(0)}e\big(  \theta\sigma_k(\mathbf{x}) -ky\theta\sigma_{k-1}(\mathbf{x})      \big)
\]
Now in this sum, by Fourier orthogonality,  the contribution of the $\mathbf{x}$ such that $\sigma_{k-1}(\mathbf{x}) =h$ is
$\ds{\int_{[0,1]^{k-1}} \left|   \Phi( \boldsymbol{\alpha}, \theta;2X )                \right|^{2b}e\big(-\alpha_{k-1}h-ky\theta h\big)     \dd\boldsymbol{\alpha}}$
where we have set
\[
\Phi( \boldsymbol{\alpha}, \theta;X ):=\sum_{|x|\le X}e\Big(P_{k-1}(\boldsymbol{\alpha}; x)+\theta x^k\Big).
\]
Due to the size of $\mathbf{x}$, we necessarily have $|h| \ll X^{k-1}$. Summing up over $h$, we have
\begin{align*}
I(\gamma,y;\theta)&=\int_{[0,1]^{k-1}} \left|     \Phi( \boldsymbol{\alpha}, \theta;2X )       \right|^{2b}\Big(\sum_{|h|\ll X^{k-1}}e\big(-\alpha_{k-1}h-ky\theta h\big)   \Big)  \dd\boldsymbol{\alpha}\\
&\ll \int_{[0,1]^{k-1}} \left|     \Phi( \boldsymbol{\alpha}, \theta;2X )   \right|^{2b}\min\Big(X^{k-1},\frac{1}{\|k\theta y+\alpha_{k-1}\|}\Big)  \dd\boldsymbol{\alpha}.\\
\end{align*}
Now inserting this estimate in \eqref{g-theta}, we have
\[
G(\theta)
\ll  (\log(4X))^{2b}  \int_{[0,1]^{k-1}} \left| \Phi( \boldsymbol{\alpha}, \theta;2X )   \right|^{2b}  \psi_k(\theta,\alpha_{k-1};X)  \dd\boldsymbol{\alpha}
\]
Integrating over $\theta$, we have
\begin{align*}
\int_AG(\theta)\dd\theta &\ll (\log(4X))^{2b}\int_{[0,1]^{k-1}\times A} \left|  \Phi( \boldsymbol{\alpha}, \theta;2X )    \right|^{2b}     \psi_k(\theta,\alpha_{k-1};X)       \dd\boldsymbol{\alpha}\dd\theta\\
&\ll  (\log(4X))^{2b}\int_{[0,1]^{k-1}\times A} \left|  \Phi( \boldsymbol{\alpha}, \theta;2X )  \right|^{2b}\psi_k^{*}(\theta;X)\dd\boldsymbol{\alpha}\dd \theta\\
&\ll    (\log(4X))^{2b}    \sup_{\theta\in A}   \psi_k^{*}(\theta;X)  \int_{[0,1]^{k-1}\times [c,c+1]} \Big|   \sum_{|x|\le 2X}e\Big(P_{k}(\boldsymbol{\alpha}; x)\Big) \Big|^{2b}\dd\boldsymbol{\alpha},\\
\end{align*}
and this last integral is $J_{b,k}(4X+1)$ by translation invariance of the Vinogradov system, which give concludes the proof.
\end{proof}
In order to estimate  $\psi_k^{*}(\theta;X)$, we recall the following classical result :

\begin{lem}\label{dist-entier} Let $\alpha,\mu \in \RR$ such that $|\alpha-\tfrac{a}{q}|\le \frac{1}{q^2}$, and let $Y,\Delta>0$. Then 
\[
\sum_{y=1}^X\min \Big(Y,\frac{1}{ \|\alpha y +\mu\|}\Big)\ll Y\big(1+\frac{X}{q}\big)+(X+q)\log Y
\]
\end{lem}

\begin{proof}
Under the assumption made on $\alpha$ and $\mu$,  Lemma 6 of \cite{HB} gives
\[
\#\{ 1\le y\le X\colon \|\alpha y +\mu\|\le \Delta \}\ll 1+X\Delta +\frac{X}{q}+q\Delta.
\]
The announced result then follows from  a dyadic summation according to the size of $\|\alpha y +\mu\|$ (see also equation (2.13) of \cite{W}).

\end{proof}

\subsection{Applications of the Beurling-Selberg function}

\begin{lem}\label{beurling} Let $k\ge 1$. let $\mathcal{E}$ be a finite set , and consider $\boldsymbol{\varphi} :\mathcal{E}\to \RR^k$. Let $(T_j)_{1\le j\le k}$,  $(T'_j)_{1\le j\le k}$ and $(\delta_j)_{1\le j\le k}$ be sequences of positive real numbers. Write $P_k=\prod_{j=1}^k[-T_j,T_j]$, $P'_k=\prod_{j=1}^k[-T'_j,T'_j]$   and $\Delta_k=\prod_{j=1}^k[-\delta_j,\delta_j]$. 
\begin{enumerate}[(i)]
\item For any sequence $\big(a(z)\big)_{z\in \mathcal{E}}$ of complex numbers of modulus at most one, one has 
\[
\int_{P_k}\left|\sum_{z\in \mathcal{E}}a(z) e(\boldsymbol{\alpha} \cdot \boldsymbol{\varphi}(z))  \right|^2   \dd\boldsymbol{\alpha}\le \left(\prod_{j=1}^k(2T_j+\frac{1}{\delta_j})\right)
\#\big\{ (z,w)\in \mathcal{E}^2\colon   \boldsymbol{\varphi}(z)- \boldsymbol{\varphi}(w)\in \Delta_k      \big\}.
\]
\item One has
\[
\#\big\{ (z,w)\in \mathcal{E}^2\colon   \boldsymbol{\varphi}(z)- \boldsymbol{\varphi}(w)\in \Delta_k \big\}\le \left(\prod_{j=1}^k(2\delta_j+\frac{1}{T_j})\right) \int_{P_k}\Big|\sum_{z\in \mathcal{E}}e(\boldsymbol{\alpha} \cdot \boldsymbol{\varphi}(z))  \Big|^2   \dd\boldsymbol{\alpha}.
\]
\item One has
\[
\int_{P_k}\left|\sum_{z\in \mathcal{E}}a(z) e(\boldsymbol{\alpha} \cdot \boldsymbol{\varphi}(z))  \right|^2   \dd\boldsymbol{\alpha}\le 8^k\Big(\prod_{j=1}^k\frac{T_j}{T'_j}\Big)  \int_{P'_k}\left|\sum_{z\in \mathcal{E}}e(\boldsymbol{\alpha} \cdot \boldsymbol{\varphi}(z))  \right|^2   \dd\boldsymbol{\alpha}.
\]
\end{enumerate}
\end{lem}

\begin{proof} The proof relies on properties of the Beurling-Selberg  function : writing $B_0:=\prod_{j=1}^k(2T_j+\tfrac{1}{\delta_j})$,  there exists a function $f\in \mathcal{L}^1(\RR^k)$ such that
\[
\mathbf{1}_{P_k}\le f\quad\mbox{and}\quad\widehat{f}\le B_0\mathbf{1}_{\Delta_k},
\]
where
\be{trans-four}
\widehat{f}(\boldsymbol{\xi})=\int_{\RR^k}f(\boldsymbol{\alpha})e\big(-\boldsymbol{\xi}\cdot \boldsymbol{\alpha}\big)\dd \boldsymbol{\alpha}\qquad (\boldsymbol{\xi}\in \RR^k)
\ee
(see \cite{vaaler}). Assertion  (i)  is essentially Lemma 7.4 of \cite{GK}. The second assertion may be derived  using the same argument (permuting the $T_j$'s and the $\delta_j$'s), and finally, (iii)  is obtained using the two first inequalities with a straightforward optimisation over the $\delta_j$'s.

\end{proof}

\subsection{An application on van der Corput's method for a polynomial phase}

The following result is a consequence of a van der Corput estimate proved in \cite{sargos}, with a new approach (see \cite{R-indag2016}).

\begin{lem}\label{vdc3} One has
\[
\sum_{|x|\le X}e\big(\alpha_1 x+\alpha_2 x^3+\alpha_3 x^5\big)   \ll \frac{X}{(1+X^5|\alpha_3|)^{1/10}}
\]
uniformly for $X\ge 1$ and $\boldsymbol{\alpha}\in \RR^3$ such that $|\alpha_3|\le X^{-10/3}$.
\end{lem}

\begin{proof} We first observe that in the case $|\alpha_3|\le X^{-5}$, the trivial bound gives the expected result. Therefore, in the sequel, we assume that $X^{-5}<|\alpha_3|\le X^{-10/3}$. We start with the upper bound
\[
\Big|  \sum_{|x|\le X}e\big(\alpha_1 x+\alpha_2 x^3+\alpha_3 x^5\big)     \Big|\ll 1+ \Big|   \sum_{1<x\le X}e\big(\alpha_1 x+\alpha_2 x^3+\alpha_3 x^5\big)       \Big|.
\]

In  the terminology of Lemma 5 of \cite{R-indag2016}, we  establish, using van der Corput's $A$-process and Corollaire 4.2 of \cite{sargos}, that $\big(\frac{1}{14},\frac{3}{7}\big)$ is a van der Corput $4$-couple. Using Lemma 5 (ii) of  \cite{R-indag2016}, this implies that for $\varphi:[U,2U]\to \RR$ defined by $\varphi(x)=\alpha_1 x+\alpha_2 x^3+\alpha_3 x^5$ for $x\in [U,2U]$, since $|\varphi^{(4)}(x)|\asymp U|\alpha_3|$ for $x\in[U,2U]$, one has
\[
\sum_{U<x\le 2U}e\big(\varphi(x)\big)\ll U^{3/5}(U|\alpha_3|)^{-1/10}
\]
uniformly for $U\le \big(U|\alpha_3|\big)^{-3/7}$. Hence, uniformly for $X^{-5}<|\alpha_3|\le X^{-10/3}$ and $1\le U\le X$, one has
\[
\sum_{U<x\le 2U}e\big(\alpha_1 x+\alpha_2 x^3+\alpha_3 x^5\big)\ll U^{1/2}|\alpha_3|^{-1/10}.
\]
Now, using the dyadic sums $\displaystyle{\sum_{1< x\le X}=\sum_{2^k\le X}\sum_{X2^{-k-1}<x\le X2^{-k}}}$ and applying the previous bound on the inner sums with the choice $U=X2^{-k-1}$, one has
\[
\sum_{0\le x\le X}e\big(\alpha_1 x+\alpha_2 x^3+\alpha_3 x^5\big)\ll 1+ X^{1/2}|\alpha_3|^{-1/10}\ll \frac{X}{(1+X^5|\alpha_3|)^{1/10}},
\]
which concludes the proof.
\end{proof}

\begin{lem}\label{alpha3-short}Let $\mathbf{k}=(1,3,5)$. Then
\[
\int_{[0,1]^2\times [c-T,c+T]}\left|  \sum_{|x|\le X}e\big(\alpha_1 x+\alpha_2 x^3+\alpha_3 x^5\big) \right|^{20}  \dd\alpha_1\dd\alpha_2 \dd\alpha_3   \ll_{\ep }X^{11+\ep}
\]
uniformly for $c\in [0,1]$ and $0\le T \le X^{-10/3}$.
\end{lem}

\begin{proof}Writing
\be{f-135}
f(\alpha_1,\alpha_2,\alpha_3;X)=\sum_{|x|\le X}e\big(\alpha_1 x+\alpha_2 x^3+\alpha_3 x^5\big),\quad (\boldsymbol{\alpha}\in \RR^3),
\ee
we have
\[
\int_{[0,1]^2\times [c-T,c+T]}\left|    f(\alpha_1,\alpha_2,\alpha_3;X)   \right|^{20}  \dd\alpha_1\dd\alpha_2 \dd\alpha_3  
\]
\[
=\int_{[0,1]^2\times [-T,T]}\left|  \sum_{|x|\le X}e(cx^5)e\big(\alpha_1 x+\alpha_2 x^3+\alpha_3 x^5\big) \right|^{20}  \dd\alpha_1\dd\alpha_2 \dd\alpha_3  
\]
\[
\ll \int_{[0,1]^2\times [-T,T]}\left|      f(\alpha_1,\alpha_2,\alpha_3)    \right|^{20}  \dd\alpha_1\dd\alpha_2 \dd\alpha_3  
\]
by using (iii) of Lemma \ref{beurling} with $T_i=T'_i=T$. Using Lemma \ref{vdc3}, since $T\le X^{-10/3}$, we have
\[
\left|    f(\alpha_1,\alpha_2,\alpha_3;X)    \right|^{10} \ll \frac{X^{10}}{1+X^5|\alpha_3|}, \quad \big(\boldsymbol{\alpha}\in [0,1]^2\times [-T,T],\thinspace 0\le T\le X^{-10/3}\big)
\]
so that
\[
\int_{[0,1]^2\times [-T,T]}\left| f(\boldsymbol{\alpha};X)   \right|^{20}  \dd\boldsymbol{\alpha}
\]
\[
\ll \int_{-T}^T   \frac{X^{10}}{1+X^5|\alpha_3|}     \left(     \int_{[0,1]^2} \left|  f(\alpha_1,\alpha_2,\alpha_3;X)   \right|^{10}  \dd\alpha_1\dd\alpha_2            \right) \dd\alpha_3  
\]
\[
\ll \int_{-T}^T   \frac{X^{10}}{1+X^5|\alpha_3|}     \left(     \int_{[0,1]^2} \left|  f(\alpha_1,\alpha_2,0;X) \right|^{10}  \dd\alpha_1\dd\alpha_2            \right) \dd \alpha_3
\]
where once again we have used  (iii)   of Lemma \ref{beurling} for the inner integral.
Now, using Lemma 5  of \cite{BR-2015} , the new inner integral is $\ll  X^{6+\ep}$.
We conclude with a direct computation of the remaining  integral over $\alpha_3$.
\end{proof}

\section{Singular integrals and singular series}

Let $s\ge 1$ and  $\mathbf{k}$ be as in \eqref{kbold}. For any $\mathbf{F}\in \mathcal{D}(\mathbf{k},s)$, we recall that $\mathfrak{I}(\mathbf{F})$ and $\mathfrak{S}(\mathbf{F})$ have been defined in \eqref{int-sing} and \eqref{serie-sing} respectively. The purpose of this section is to prove that for $s$ sufficiently large, these constants are positive, as soon as \eqref{syst1}  has nonsingular solution over $\RR$ and the $p$-adic, in view of the asymptotic in Theorems \ref{circle2} and \ref{circle3}. For both constants, we follow quite closely the lines of \cite{SP} and \cite{schmidt}.

\subsection{Singular integrals}

\begin{lem}\label{prop-intsing}
Let $\mathbf{k}$ be as in \eqref{kbold}, $s\ge nk_n+1$ and $\mathbf{F}\in \mathcal{D}(\mathbf{k},s)$. Then $\mathfrak{I}(\mathbf{F})$  is absolutely convergent.  Moreover, if  the system  \eqref{syst1} has a  nonsingular solution over $\RR$, then $\mathfrak{I}(\mathbf{F})>0$.
\end{lem}

\begin{proof}
For any $T\ge 1$ and any $\boldsymbol{\beta}\in \RR^n$ we set
\[
w_T(\boldsymbol{\beta})=\prod_{i=1}^n\frac{\sin^2\big(\pi \beta_i/T\big)}{(\pi\beta_i/T)^2}\qquad (\boldsymbol{\beta}\neq \mathbf{0}),\qquad w_T(\mathbf{0})=1.
\]
Classically, for any $\mathbf{y}\in \RR^n$ we have
\[
\widehat{w_T}(\mathbf{y}):=\int_{\RR^n}w_T(\boldsymbol{\beta})e\big(\boldsymbol{y}\cdot\boldsymbol{\beta}\big)\dd \boldsymbol{\beta}=T^n\prod_{i=1}^n\max\big(0,1-T|y_i|\big)
\]

Using Lemma \ref{maj-int-osc}, we have
\[
v(\mathbf{F}; \boldsymbol{\beta}):=   \int_{[-1,1]^s}  e\big(\boldsymbol{\beta}\cdot \mathbf{F}(\mathbf{t})\big)\dd\mathbf{t}   \ll_{\mathbf{F}}\Big(1+\sum_{i=1}^n|\beta_i|\Big)^{-s/k_n}\qquad (\boldsymbol{\beta}\in \RR^n),
\]
hence (ii) of  Lemma \ref{maj-int-2} implies
$\ds{\int_{\RR^n}|v(\mathbf{F}; \boldsymbol{\beta})|\dd \boldsymbol{\beta}<+\infty}$
since $s>nk_n$.
Setting
\be{}
\mathfrak{I}_T(\mathbf{F})=\int_{\RR^n}  w_T(\boldsymbol{\beta})v(\mathbf{F}; \boldsymbol{\beta})  \dd \boldsymbol{\beta}\qquad (T\ge 1)
\ee 
it follows easily from Lebesgue's theorem that
$\ds{\lim_{T\to +\infty}\mathfrak{I}_T(\mathbf{F})=\mathfrak{I}(\mathbf{F})}$.
Moreover, using Fubini's theorem, one can  easily deduce that
\be{}
\mathfrak{I}_T(\mathbf{F})=\int_{[-1,1]^s}\widehat{w_T}\big(\mathbf{F}(\mathbf{t})\big)\dd\mathbf{t}\ge 0\qquad (T\ge 0).
\ee

By homogeneity of $\mathbf{F}$, we may assume that system \eqref{syst1} has a nonsingular  solution $\boldsymbol{\eta}\in \big[-\tfrac{1}{2},\tfrac{1}{2}\big]^s$. Up to a renumbering of the coordinates $\eta_1,\dots,\eta_s$, we may assume that the matrix
$\ds{D\mathbf{F}(\boldsymbol{\eta})=\Big(\frac{\partial F_i}{\partial t_j}(\boldsymbol{\eta})\Big)_{1\le i,j\le n}}$
has maximal rank. Consider the map $\psi:\RR^s\to \RR^s$ defined by
\[
\psi(\mathbf{t})=(F_1(\mathbf{t}), F_2(\mathbf{t}), \dots,F_n(\mathbf{t}),t_{n+1},t_{n+2},\dots,t_s)\qquad (\mathbf{t}\in \RR^s).
\]
Writing $J_{\psi}(\mathbf{t})$ the jacobian of $\psi$ at $\mathbf{t}$, one has
$\det J_{\psi}(\boldsymbol{\eta})=\det D\mathbf{F}(\boldsymbol{\eta})\neq 0$. Hence,
the Inverse Function Theorem implies that for some open neighbourhood $U_0\subset [-1,1]^s$ of $\boldsymbol{\eta}$, some open neighbourhood $V_0\subset \RR^n$ of $\mathbf{0}$  and some open neighbourhood  $W_0\subset \RR^{s-n}$ of $(\eta_{n+1},\eta_{n+2},\dots,\eta_s)$, the map $\psi:U_0\to V_0\times W_0$ is a $C^1$ diffeomorphism.
Moreover, there exists $T_0\ge 1$ such that 
$\big[-\tfrac{1}{T_0},\tfrac{1}{T_0}\big]^n\subset V_0$
and some open neighbourhood $W_1$ of    $(\eta_{n+1},\eta_{n+2},\dots,\eta_s)$  such that  $\overline{W_1}\subset W_0$.

We set
\[
K_T:=\psi^{-1}\Big(   \big[-\tfrac{1}{T},\tfrac{1}{T}\big]^n  \times   \overline{W_1}    \Big)\qquad (T\ge T_0).
\]
For some  $C_0>0$ we have
$\ds{\frac{1}{C_0}\le |\det J_{\psi}(\boldsymbol{t})|\le C_0}$ whenever $\mathbf{t}\in K_{T_0}$.
Now, for $T\ge T_0$, one has 
\[
\mathfrak{I}_T(\mathbf{F})\ge \int_{K_T}\widehat{w_T}\big(\mathbf{F}(\mathbf{t})\big)\dd\mathbf{t}
\ge \frac{1}{C_0}\int_{K_T}\widehat{w_T}\big(\mathbf{F}(\mathbf{t})\big)|\det J_{\psi}(\boldsymbol{t})|      \dd\mathbf{t}.
\]
Using the change of variables $\mathbf{y}=\psi(\mathbf{t})$, this last integral is equal to
\[
\int_{\psi(K_T)}\widehat{w_T}(y_1,\dots,y_n)\dd \mathbf{y}=\mathrm{Meas}\big(    \overline{W_1}  \big)\int_{\big[-\frac{1}{T},\frac{1}{T}\big]^n}\widehat{w_T}(y_1,\dots,y_n)\dd y_1\dots \dd y_n.
\]
By a simple computation, this last integral is equal to  $1$. Hence, one has
\[
\mathfrak{I}_T(\mathbf{F})\ge \frac{   \mathrm{Meas}\big(    \overline{W_1}  \big)   }{C_0} \qquad (T\ge T_0).
\]
Letting $T$ tend to $+\infty$, one has 
$\ds{\mathfrak{I}(\mathbf{F})\ge   \frac{   \mathrm{Meas}\big(    \overline{W_1}  \big)   }{C_0}>0}$,
which is the expected result.
\end{proof}
\subsection{Singular series}

\begin{lem}\label{prop-seriesing}
Let $\mathbf{k}$ be as in \eqref{kbold}, $s\ge (n+1)k_n+1$ and $\mathbf{F}\in \mathcal{D}(\mathbf{k},s)$. Then $\mathfrak{S}(\mathbf{F})$  is absolutely convergent.  Moreover, if  the system  \eqref{syst1}   has a nonsingular solution over each $p$-adic $\QQ_p$, then $\mathfrak{S}(\mathbf{F})>0$.

\end{lem}
\begin{proof}
We set
\be{}
T(q)=\frac{1}{q^s}\sum_{\mathbf{a}\in A_n(q)}\sum_{\mathbf{x}\in [1,q]^n}e\Big(\frac{\mathbf{a}\cdot \mathbf{F}(\mathbf{x})}{q}\Big)\qquad (q\ge 1).
\ee
Writing the sum over $\mathbf{x}$ as in \eqref{sum-r-mult} and      using  (iii) of Lemma \ref{som-complet}, we have the estimate
$T(q)\ll_{\mathbf{k},\ep} q^{n-\frac{s}{k_n}+\ep}$ $(q\ge 1)$
so that for $s>(n+1)k_n$ the series
$\mathfrak{S}(\mathbf{F})$
is absolutely convergent. We now recall that $T(q)$ is multiplicative, \textit{i.e.} that $T(qq')=T(q)T(q')$ whenever $q$ and $q'$ are coprime.The proof is quite similar to that of \cite{SP}. We omit the details. We now have
\be{prod-cv}
\mathfrak{S}(\mathbf{F})=\prod_{p}\Big(1+\sum_{h\ge 1}T(p^{h})\Big)
\ee
where this product is absolutely convergent. Moreover, for each $p\ge 2$, one has  
\be{serie-p}
1+\sum_{h \ge 1}T(p^{h})=\lim_{H\to+\infty }p^{H(n-s)}M(p^H)
\ee
where $M(q)$ is the number of solutions of \eqref{syst1} in $(\ZZ/q\ZZ)^s$. For $p\ge 2$ fixed, we assumed that \eqref{syst1} has a nonsingular solution  $\boldsymbol{\eta}\in \ZZ_p^s$. Up to a renumbering of the coordinates $\eta_1,\dots,\eta_s$,
we may assume that  
$\ds{D\mathbf{F}(\boldsymbol{\eta})=\Big(\frac{\partial F_i}{\partial t_j}(\boldsymbol{\eta})\Big)_{1\le i,j\le n}}$
has maximal rank. 
We set $\ds{ F_i^{(1)}(\mathbf{t})=\sum_{j=1}^nu_{i,j}t_j^{k_i}     }$   and    $\ds{  F_i^{(2)}(\mathbf{t})=\sum_{j=n+1}^su_{i,j}t_j^{k_i}       }$ for  $1\le i\le n$
so that  we have   $\mathbf{F}=\mathbf{F}^{(1)}+\mathbf{F}^{(2)}$ where $\mathbf{F}^{(j)}=(F_1^{(j)},F_2^{(j)},\dots,F_n^{(j)})$   $(j=1,2)$. 
Hence, with this notation we have
$\mathbf{F}^{(1)}(\eta_1,\dots,\eta_n)+\mathbf{F}^{(2)}(\eta_{n+1},\dots,\eta_s)=0$.
We recall that $\det D\mathbf{F}(\boldsymbol{\eta})\neq 0$ and  let $v_p$ be its $p$-adic valuation. For $u_p:=2v_p+1$, we  have the following : for any fixed $(\mu_{n+1},\dots,\mu_s)$ such that
\be{mu-congru}
(\mu_{n+1},\dots,\mu_s)\equiv  (\eta_{n+1},\dots,\eta_s)\mod p^{u_p},
\ee
we have 
$\mathbf{F}^{(1)}(\eta_1,\dots,\eta_n)   +\mathbf{F}^{(2)}(\mu_{n+1},\dots,\mu_s)          \equiv \mathbf{0}\mod p^{u_p}$.
From this,  Hensel's Lemma asserts that     $(\eta_1,\dots,\eta_n)$      lifts to a unique $(\mu_1,\dots,\mu_n)\in \ZZ_p^n$ such that
\[
\mathbf{F}^{(1)}(\mu_1,\dots,\mu_n)   +\mathbf{F}^{(2)}(\mu_{n+1},\dots,\mu_s)=0
\]
with $(\mu_1,\dots,\mu_n)\equiv (\eta_1,\dots,\eta_n)\mod p^{v_p+1}$.
Finally, for any $H\ge u_p$, there are at least $p^{(H-u_p)(s-n)}$ choices of $(\mu_{n+1},\dots,\mu_s)\in (\ZZ/p^H\ZZ)^{s-n}$ that satisfy \eqref{mu-congru}, and each of them contributes for at least one solution of \eqref{syst1} in $(\ZZ/p^H\ZZ)^{s}$. Hence
\[
M(p^{H})\ge p^{(H-u_p)(s-n)}.
\]
Inserting this last inequality into \eqref{serie-p}, the corresponding series is also $\ge p^{-u_p(s-n)}$,
so that each  of the factors in  \eqref{prod-cv} is positive.     Since the product \eqref{prod-cv} is absolutely convergent, this implies that $\mathfrak{S}(\mathbf{F})>0$.
\end{proof}
\section{Estimates related to major arcs}

We start with the definition of the major arcs and the minor arcs for our problem.

Let $\mathbf{k}$ be fixed as in \eqref{kbold}, and $\tau$ fixed such that $\frac{1}{nk_n}\le \tau\le 1$. For $X$ sufficiently large, writing 
\be{defQ}
Q=\lfloor X^{\tau}\rfloor,
\ee
the set of major arcs is 
\be{majorarc1}
\mathfrak{M}=\mathfrak{M}(X)=\bigcup_{q\le Q}\bigcup_{\mathbf{a}\in A_n(q)}\mathfrak{M}(q,\mathbf{a})
\ee
where we have set
\be{majorarc2}
\mathfrak{M}(q,\mathbf{a})=\prod_{i=1}^n\left[\frac{a_i}{q}- \frac{Q}{qX^{k_i}}  , \frac{a_i}{q}+ \frac{Q}{qX^{k_i}}  \right].
\ee
Writing
\be{defQ0}
Q_0=2Q,
\ee
one has $\mathfrak{M}\subset \Big[\frac{1}{Q_0},1+\frac{1}{Q_0}\Big]^n$. The set of minor arcs is
\be{minorarc}
\mathfrak{m}=\Big[\frac{1}{Q_0},1+\frac{1}{Q_0}\Big]^n\smallsetminus \mathfrak{M}.
\ee

We are now ready to state our estimate for the major arcs.
\begin{thm}\label{thm-major}
Let $\mathbf{k}$ as in \eqref{kbold}, $\tfrac{1}{nk_n}\le \tau\le 1$ and $\mathfrak{M}$ as in \eqref{majorarc1}. Then for any $s\ge (n+1)k_n+1$ and any   $\mathbf{F}\in  \mathcal{D}(\mathbf{k},s)$, we have
\[
\int_{\mathfrak{M}}
\Big(\sum_{\mathbf{x}\in I_s(X)}e\big(\boldsymbol{\alpha}\cdot \mathbf{F}(\mathbf{x})\big)\Big)\dd \boldsymbol{\alpha}=
\mathfrak{I}(\mathbf{F})\mathfrak{S}(\mathbf{F})X^{s-\sigma(\mathbf{k})}+O(   X^{s-\sigma(\mathbf{k})-\frac{\tau}{k_n}+\ep})\qquad (X\ge 1).
\]
\end{thm}

\begin{proof}  Throughout this proof, the quantities   $f[\mathbf{F}](\boldsymbol{\alpha},X)$,    $S[\mathbf{F}](q,\mathbf{a})$,     $v[\mathbf{F}](\boldsymbol{\beta},X)$,   $\xi_{\mathbf{k}}(\boldsymbol{\beta}, X)$ defined    in   equations    \eqref{sum-weyl-multi} to \eqref{xi-k}  are  written    more simply   $f(\boldsymbol{\alpha},X)$,  $S(q,\mathbf{a})$,     $v(\boldsymbol{\beta},X)$ and   $\xi(\boldsymbol{\beta}, X)$.   Writing $I(\mathfrak{M})$ for the integral over the major arcs, using that $\mathfrak{M}$ is a disjoint union, we have
\[
I(\mathfrak{M})
=\sum_{q\le Q}\sum_{\mathbf{a}\in A_n(q)}\int_{\mathfrak{M}(q,\mathbf{a})}  f(\boldsymbol{\alpha},X)
\dd \boldsymbol{\alpha}
=\sum_{q\le Q}\sum_{\mathbf{a}\in A_n(q)}\int_{\mathfrak{M}(q,\mathbf{0})}     f\Big(\frac{\mathbf{a}}{q}+\boldsymbol{\beta},X\Big)
\dd \boldsymbol{\beta}.
\]
Now inserting the estimate from Lemma \ref{weyl-mult} in each of these right hand side  integrals,
we have $I(\mathfrak{M})=I_1(\mathfrak{M})+O\big(I_2(\mathfrak{M})+ I_3(\mathfrak{M})    \big)$
where we have set
\begin{align*}
I_1(\mathfrak{M})&=\sum_{q\le Q}\sum_{\mathbf{a}\in A_n(q)}\int_{\mathfrak{M}(q,\mathbf{0})}
\left(            \frac{X^s}{q^s}  S(q,\mathbf{a})  v(\boldsymbol{\beta},X)                                \right)\dd \boldsymbol{\beta},\\
I_2(\mathfrak{M})&=\sum_{q\le Q}\sum_{\mathbf{a}\in A_n(q)}\int_{\mathfrak{M}(q,\mathbf{0})}
\Big(      X^{s-1}q^{1-\frac{s}{k_n}+\ep}    \xi(\boldsymbol{\beta}, X)^{-(s-1)/k_n}\Big)\dd \boldsymbol{\beta},\\
I_3(\mathfrak{M})&=\sum_{q\le Q}\sum_{\mathbf{a}\in A_n(q)} q^{s-\frac{s}{k_n}+\ep} \int_{\mathfrak{M}(q,\mathbf{0})}   \dd \boldsymbol{\beta}.\\
\end{align*}
We already have
\[
I_3(\mathfrak{M})\ll \sum_{q\le Q}\sum_{\mathbf{a}\in A_n(q)} q^{s-\frac{s}{k_n}+\ep}\frac{Q^n}{q^nX^{\sigma(\mathbf{k})}}\ll \frac{Q^{s+n+1-\frac{s}{k_n}+\ep}}{  X^{\sigma(\mathbf{k})}   }\ll X^{s-\sigma(\mathbf{k})-\frac{\tau}{k_n}+\ep}
\]
by using the bounds $Q^{s+\ep}\ll X^{s+\ep}$ and $Q^{n+1-\frac{s}{k_n}}\ll Q^{-1/k_n}=X^{-\tau/k_n}$, where for the first bound, we have used the fact that $\tau\le 1$, and for the last bound, we have used the inequality $n+1-s/k_n\le -1/k_n$.
Now, using the change of variables  $\gamma_i=X^{k_i}\beta_i$ in the integrals of  $I_1(\mathfrak{M})$ and  $I_2(\mathfrak{M})$, we have
\[
I_1(\mathfrak{M})=X^{s-\sigma(\mathbf{k})}\sum_{q\le Q}\frac{1}{q^s}\sum_{\mathbf{a}\in A_n(q)}    S(q,\mathbf{a})  \int_{[-\frac{Q}{q},\frac{Q}{q}]^n}            v(\boldsymbol{\gamma},1)       \dd \boldsymbol{\gamma}
\]
and
\[
I_2(\mathfrak{M})=X^{s-1-\sigma(\mathbf{k})}\sum_{q\le Q}\sum_{\mathbf{a}\in A_n(q)} q^{1-\frac{s}{k_n}+\ep}         \int_{  [-\frac{Q}{q},\frac{Q}{q}]^n     }
      \xi(\boldsymbol{\gamma}, 1)^{-(s-1)/k_n}\dd \boldsymbol{\gamma}.
\]

Using Lemma \ref{maj-int-osc}  and (ii) of Lemma \ref{maj-int-2}, the inner integrals in $I_1(\mathfrak{M})$ satisfy
\[
\mathfrak{J}(\mathbf{F})-   \int_{[-\frac{Q}{q},\frac{Q}{q}]^n} v(\boldsymbol{\gamma},1) \dd \boldsymbol{\gamma} \ll \left(Q/q\right)^{n-s/k_n} \qquad (1\le q\le Q).
\]
Hence
\[
I_1(\mathfrak{M})=X^{s-\sigma(\mathbf{k})}\sum_{q\le Q}\frac{1}{q^s}\sum_{\mathbf{a}\in A_n(q)}      S(q,\mathbf{a})          \left(  \mathfrak{J}(\mathbf{F})+O\big(     \left( Q/q\right)^{n-s/k_n}     \big)         \right).
\]
Using   \eqref{sum-r-mult} and (i) of Lemma \ref{som-complet}, we have 
\[
I_1(\mathfrak{M})= X^{s-\sigma(\mathbf{k})}    \mathfrak{J}(\mathbf{F})  \sum_{q\le Q}\frac{1}{q^s}\sum_{\mathbf{a}\in A_n(q)}   S(q,\mathbf{a}) +O(I_4)
\]
where
\[
I_4=   X^{s-\sigma(\mathbf{k})} \sum_{q\le Q}\frac{1}{q^s}\sum_{\mathbf{a}\in A_n(q)}q^{s-\frac{s}{k_n}+\ep}    \left( Q/q\right)^{n-s/k_n}
\ll X^{s-\sigma(\mathbf{k})-\frac{\tau}{k_n}+\ep}
\]
by using the same inequalities as for $I_3(\mathfrak{M})$.
Moreover, using again \eqref{sum-r-mult} and (i) of Lemma \ref{som-complet}, we have
\[
\mathfrak{S}(\mathbf{F})- \sum_{q\le Q}\frac{1}{q^s}\sum_{\mathbf{a}\in A_n(q)}     S(q,\mathbf{a}) 
\ll \sum_{q>Q}\frac{1}{q^s}\sum_{\mathbf{a}\in A_n(q)}q^{s-\frac{s}{k_n}+\ep}\ll X^{-\frac{\tau}{k_n}+\ep}
\]
which gives 
$I_1(\mathfrak{M})=\mathfrak{J}(\mathbf{F}) \mathfrak{S}(\mathbf{F}) X^{s-\sigma(\mathbf{k})} +O\left(   X^{s-\sigma(\mathbf{k})-\frac{\tau}{k_n}+\ep}  \right)$.
Finally, using the same inequalities, we have
\[
I_2(\mathfrak{M})\ll X^{s-1-\sigma(\mathbf{k})}\sum_{q\le Q}\sum_{\mathbf{a}\in A_n(q)} q^{1-\frac{s}{k_n}+\ep} \left(1+\tfrac{Q}{q}\right)^{1/k_n}\ll X^{s-\sigma(\mathbf{k})-\frac{\tau}{k_n}+\ep},
\]
which completes the proof.
\end{proof}

\section{Classical minor arcs estimates}

The following result is merely  Theorem 5.2 of \cite{vaughan} applied to the sum $f_{\mathbf{k}}(\boldsymbol{\alpha};X)$ defined in \eqref{fk}.

\begin{prop}\label{vaughan-minor}
Let $n\ge 2$, $\mathbf{k}$ as in \eqref{kbold} and $b\ge 1$. Let $\boldsymbol{\alpha}\in \RR^n$. Suppose that there exist $j,a_j,q_j$ with $k_j\ge 2$, $|\alpha_j-\frac{a_j}{q_j}|\le \frac{1}{q_j^2}$, $(a_j;q_j)=1$, $q_j\le X^{k_j}$. Then, with the notation \eqref{int-vino}, one has
\[
f_{\mathbf{k}}(\boldsymbol{\alpha};X)\ll_{b,n,\mathbf{k}} \big(X^{k_n(k_n-1)/2}J_{b,k_n-1}(2X)\big)^{1/(2b)}\Big(\frac{q_j}{X^{k_j}}+\frac{1}{X}+\frac{1}{q_j}\Big)^{1/(2b)}\log(2X)
\]
\end{prop}


In order to treat the minor arcs for an asymptotic for the Vinogradov-type system, Proposition \ref{vaughan-minor} and some analogue  of our Theorem \ref{weyl-asymp} are sufficient to derive a bound of the form 
\[
\sup_{\boldsymbol{\alpha}\in \mathfrak{m}}|f_{\mathbf{k}}(\boldsymbol{\alpha};X)|\ll X^{1-\rho_0}\qquad (X\ge X_0)
\]
for some $\rho_0>0$.  In the case of our system \eqref{syst1}, we shall require an analogue for exponential sums of the form $f_{\mathbf{k}}(\mathbf{w}\otimes \boldsymbol{\alpha};X)$ for some fixed
$\mathbf{w}\in \ZZ^n$ with $w_1\dots w_n\neq 0$. Although it is no trouble to derive an analogue with the same tools, in order to ease our presentation and set some notation, we state  a lemma that produces a suitable approximation for $\boldsymbol{\alpha}$ from an approximation of $\mathbf{w}\otimes \boldsymbol{\alpha}$.

\begin{lem}\label{approx-dio}     Let $\mathbf{w}\in \ZZ^n$ fixed such that $w_1w_2\dots w_n\neq 0$. Set $M_0:=|w_1w_2\dots w_n|$.
\begin{enumerate}[(i)]
\item For any $q\ge 1$ and  any $\mathbf{a}\in \ZZ^n$ with $(\mathbf{a};q)=1$, there exists $\mathbf{c}=\mathbf{c}(q,\mathbf{a},\mathbf{w})\in \ZZ^n$ and $h=h(q,\mathbf{a},\mathbf{w})\ge 1$ unique such that  
$\ds{\tfrac{\mathbf{w}\otimes \mathbf{a}}{q}=\tfrac{ \mathbf{c} }{   h }}$ with  $(\mathbf{c};h)=1$. 
Moreover one has $\frac{q}{M_0}\le h(q,\mathbf{a},\mathbf{w})\le q$.
\item Let $\mathbf{k}$ be as in \eqref{kbold}, and $\boldsymbol{\lambda}\in \RR^n$ such that $\lambda_i>0$ ($1\le i\le n$) and $\sigma:=\sum_{1\le i\le n}\lambda_i<1$. Suppose that  for any $i$ with $k_i\ge 2$ there exists $b_i\in \ZZ$, $q_i\ge 1$ coprime such that
\be{walpha-approx}
\left|w_i\alpha_i-\tfrac{b_i}{q_i}\right|\le \tfrac{X^{\lambda_i}}{q_iX^{k_i}},\qquad 1\le q_i\le X^{\lambda_i}.
\ee
Then, there exists $X_0=X_0(\boldsymbol{\lambda},\mathbf{w})$ such that  whenever $X\ge X_0$, there exists $q\in \NN$  with  $q\le M_0X^{\sigma}$,  $a_1,a_2,\dots,a_n\in \ZZ$ unique  with  $(q;a_1;a_2;\dots;a_n)=1$  such that, writing $  \boldsymbol{\beta}   =   \boldsymbol{\alpha}-   \frac{\mathbf{a}}{q}  $,
we have the following :
\begin{itemize}
\item If $k_1=1$, then $\displaystyle{|\beta_1|\le \tfrac{1}{2|w_1|h(q,\mathbf{a},\mathbf{w})  },\quad    \left|\beta_i\right|\le \tfrac{M_0X^{\sigma}}{qX^{k_i}}  \quad (2\le i\le n)}$,
\item If $k_1\ge 2$, then  $\displaystyle{\left|\beta_i\right|\le \tfrac{M_0X^{\sigma}}{qX^{k_i}}  \quad (1\le i\le n)}$,
\end{itemize}
and such that  $\mathbf{w}\otimes \boldsymbol{\beta}$ satisfies \eqref{cond-beta} with the choice $\mathfrak{q}=h(q,\mathbf{a},\mathbf{w})$.
\end{enumerate}
\end{lem}

\begin{proof} The proof of (i) only use classical divisibility properties : we omit the details.
For (ii), due to  the constraints over the $\boldsymbol{\lambda}$ and $\mathbf{k}$ and the size of the $q_i$, it is plain that for $X$ sufficiently large, the $b_i,q_i$ in \eqref{walpha-approx} are unique.
We start with the case $k_1=1$.  Using  \eqref{walpha-approx}, there exist $q'$ and $c_i$ ($2\le i\le n$)  unique  such that  $\frac{b_i}{q_iw_i}=\frac{c_i}{q'}$  with   $(c_2;c_3;\dots;c_n;q')=1$. Then we have  $q'\le X^{\sigma}$  and 
\be{walpha-approx2}
\left|\alpha_i-\tfrac{c_i}{q'}\right|\le \tfrac{X^{\sigma}}{q'X^{k_i}} \qquad (2\le i\le n).
\ee
Next, there we choose $c_1$ minimal  such that 
$\displaystyle{\big|w_1\alpha_1-\tfrac{c_1}{q'}\big|\le\frac{1}{2q'}}$. Writing now $\mathbf{c}=(c_1, c_2,\dots,c_n)$, we still have $(\mathbf{c};q')=1$.  Similarly,  there exist $q\ge 1$ and $\mathbf{a}\in \ZZ^n$ unique with
$(\mathbf{a};q)=1$  such that $\frac{c_i}{w_iq'}=\frac{a_i}{q}$ for $1\le i\le n$. As previously, we have  $q\le M_0X^{\sigma}$    and $\ds{ \Big|\alpha_1-\tfrac{a_1}{q}\Big|\le \tfrac{1}{2q'|w_1|}}$,   $\ds{ \Big|\alpha_j-\tfrac{a_j}{q}\Big|\le \tfrac{X^{\sigma}M_0}{qX^{k_j}}}$    $(2\le j\le n)$.
By unicity, it is now plain that $q'=h(q,\mathbf{a},\mathbf{w})$, which gives the expected result.

\bigskip
The case $k_1\ge 2$  is more straightforward : the construction leading to  \eqref{walpha-approx2}, initially valid for $2\le i\le n$ is now also valid for $i=1$, hence  the choice $\mathbf{a}=\mathbf{c}$ and $q=q'$  is sufficient to have the expected result. 

\bigskip
Finally, since $q\ge h(q,\mathbf{a},\mathbf{w})$, it is a simple observation  that for $X$ sufficiently large,  $\mathbf{w}\otimes \boldsymbol{\beta}$ satisfies \eqref{cond-beta} with the choice $\mathfrak{q}=h(q,\mathbf{a},\mathbf{w})$.
\end{proof}

We can now state our first result for the minor arcs defined in \eqref{minorarc}.

\begin{lem}\label{lem-minor}Let $n\ge 2$ and $\mathbf{k}$ fixed as in \eqref{kbold}. Let $\mathbf{w}\in \ZZ^n$ fixed such that all $w_1w_2\dots w_n\neq 0$. Set $\eta_0=\frac{1}{nk_n^2}$.  With the notation \eqref{fk}, we have
\[
f_{\mathbf{k}}(\mathbf{w} \otimes \boldsymbol{\alpha};X)\ll_{\mathbf{w},\ep } X^{1-\eta_0+\ep}
\]
uniformly for $\boldsymbol{\alpha}\in \mathfrak{m}$ and $X\ge X_0(\mathbf{w})$.
\end{lem} 

\begin{proof}

We set $\lambda_i=\frac{k_n-1}{nk_n}$ for $1\le i\le n$. For any $i$ such that $k_i\ge 2$,  there exist $b_i,q_i$ with $1\le q_i\le X^{k_i-\lambda_i}$ such that
$\displaystyle{\left|w_i\alpha_i-\tfrac{b_i}{q_i}\right|\le \tfrac{X^{\lambda_i}}{q_iX^{k_i}}}$.
If  one has $q_i>X^{\lambda_i}$ for some of these  $i$, then Proposition \ref{vaughan-minor} used with $b=k_n(k_n-1)/2$ and the bound \eqref{main-vino} implies
\[
f_{\mathbf{k}}(\mathbf{w} \otimes \boldsymbol{\alpha};X)\ll_{\mathbf{w},\ep }X^{1+\ep} \big(X^{-\lambda_i}\big)^{1/(2b)}\ll X^{1-\eta_0+\ep}.
\]
We may now  assume that for any   $1\le i\le n$  such that $k_i\ge 2$ we have $q_i\le X^{\lambda_i}$. Using Lemma \ref{approx-dio} and its notation, we have $\sigma=1-\frac{1}{k_n}$, $q\ll X^{\sigma}$, and for $X$ sufficiently large, $\mathbf{w}\otimes \boldsymbol{\beta}$ satisfies  \eqref{cond-beta} for some $\mathfrak{q}\asymp q$. Hence  Theorem \ref{weyl-asymp} yields 
\be{fk-bound2}
 f_{\mathbf{k}}(\mathbf{w} \otimes \boldsymbol{\alpha};X)       \ll_{\mathbf{w},\mathbf{k},\ep} X^{1+\ep}\Big( q\xi_{\mathbf{k}}(\boldsymbol{\beta},X) \Big)^{-1/k_n}+ q^{1-\frac{1}{k_n}+\ep}
\ee
where for the last   term  the bound  $q\ll X$ is sufficient. Moreover, with the notation \eqref{xi-k}   we have
$q\xi_{\mathbf{k}}(\boldsymbol{\beta},X) \gg X^{\tau}$
uniformly for $\boldsymbol{\alpha}\in \mathfrak{m}$. This combined with  \eqref{fk-bound2} yields the announced upper bound.

\end{proof}

\section{Proof of Theorem 1}

\subsection{Full-saving index for $(k_1,\dots,k_n)$}

In order to establish Theorem \ref{circle2}, we shall prove a more general result.  The starting point  is to work with a case where  \eqref{heuristic-gene} is verified. We  shall say that a number $A$ is a \textsl{full-saving index} for $\mathbf{k}$ if for any $\ep>0$ one has
\be{full-saving}
\int_{[0,1]^n}\left|    f_{\mathbf{k}}(\boldsymbol{\alpha};X)   \right|^{A}\dd \boldsymbol{\alpha}\ll_{\ep} X^{A-\sigma(\mathbf{k})+\ep}\qquad (X\ge 1).
\ee

\begin{thm}\label{circle1} Let $n\ge 2$ and $\mathbf{k}$ as in \eqref{kbold}. Let $A$ be a full saving index  for $\mathbf{k}$. Let $s\ge 1+\max \big(1+ A, (n+1)k_n\big)$  and  $\mathbf{F}\in \mathcal{D}(\mathbf{k},s)$. Then with the notation \eqref{int-sing} and \eqref{serie-sing}, both $\mathfrak{I}(\mathbf{F})$ and  $\mathfrak{S}(\mathbf{F})$  are convergent, and  one has, for any $\ep>0$ fixed
\[
N_{\mathbf{F}}(X)=\mathfrak{I}(\mathbf{F})\mathfrak{S}(\mathbf{F})X^{s-\sigma(\mathbf{k})}+O(   X^{s-\sigma(\mathbf{k})-\eta_0+\ep})\qquad (X\ge 1),
\]
where we have set $\eta_0=\frac{1}{nk_n^2}$.
If moreover the system \eqref{syst1} has a non singular solution over $\RR$ and over $\ZZ_p$ for all $p$ , then  $\mathfrak{I}(\mathbf{F})\mathfrak{S}(\mathbf{F})>0$.
\end{thm}

It is clear from \eqref{bound-full}  that $A=k_n(k_n+1)$ is a full-saving index for $\mathbf{k}$. Hence Theorem \ref{circle2} follows immediately from Theorem \ref{circle1}.

\subsection{Proof of Theorem \ref{circle1}}

 We already have the suitable estimate on the major arcs with Theorem \ref{thm-major}. We follow the classic approach described in \S 3.4 of \cite{LP}.
First, under our assumptions,  Lemmas \ref{prop-intsing} and \ref{prop-seriesing} take care of  the singular constants. We may now estimate the contribution of the minor arcs.
Using the notation in \eqref{sum-weyl-multi} and H\"older's inequality, one has
\be{minor-holder}
\int_{\mathfrak{m}}  f[\mathbf{F}](\boldsymbol{\alpha};X)          \dd \boldsymbol{\alpha}\ll \prod_{j=1}^{s}\left(  
\int_{\mathfrak{m}}        \big|f_{\mathbf{k}}(\mathbf{u}_j\otimes \boldsymbol{\alpha};X)\big|^s\dd \boldsymbol{\alpha}              \right)^{1/s}.
\ee

Hence, it is sufficient to establish the upper bound
\[
\int_{\mathfrak{m}}        \big|f_{\mathbf{k}}(\mathbf{w}\otimes \boldsymbol{\alpha};X)\big|^s\dd \boldsymbol{\alpha}  \ll  X^{s-\sigma(\mathbf{k})-\eta_0+\ep}
\]
for any $\mathbf{w}\in \ZZ^n$ with $w_i\neq 0$  ($1\le i\le n$). For such a $\mathbf{w}$, we have
\[
\int_{\mathfrak{m}}        \big|f_{\mathbf{k}}(\mathbf{w}\otimes \boldsymbol{\alpha};X)\big|^s\dd \boldsymbol{\alpha}\le \sup_{\boldsymbol{\alpha}\in \mathfrak{m}}\big|f_{\mathbf{k}}(\mathbf{w}\otimes \boldsymbol{\alpha};X)\big|\int_{[\frac{1}{Q_0},1+\frac{1}{Q_0}]^n}        \big|f_{\mathbf{k}}(\mathbf{w}\otimes \boldsymbol{\alpha};X)\big|^{s-1}\dd \boldsymbol{\alpha}
\]
\[
\ll    X^{1-\eta_0+\ep}       \int_{[0,1]^n}        \big|f_{\mathbf{k}}(\boldsymbol{\alpha};X)\big|^{s-1}\dd \boldsymbol{\alpha}
\]
by using Lemma \ref{lem-minor}. Moreover, since $s-1\ge A$, then 
\[
\int_{[0,1]^n}        \big|f_{\mathbf{k}}(\boldsymbol{\alpha};X)\big|^{s-1}\dd \boldsymbol{\alpha}\ll X^{s-1-\sigma(\mathbf{k})+\ep},
\]
which concludes the proof of Theorem \ref{circle1}.

\section{Refined estimate on the minor arcs }

In this section,  $\mathbf{k}$ is fixed equal to $(1,3,5)$. We use the notation    $f(\alpha_1,\alpha_2,\alpha_3;X)$ introduced in  \eqref{f-135}, and  we consider the major arcs introduced in   \eqref{majorarc1}     with the choice $\tau=5/8$.
To prove Theorem \ref{circle3}, it is clear, using \eqref{minor-holder}, that  it is now sufficient to prove the following :

\begin{thm} \label{th-minor2} Let $\mathbf{w}=(w_1,w_2,w_3)\in \ZZ^3$ with $w_1w_2w_3\neq 0$.  With the notation \eqref{f-135}, we have
\[
\int_{\mathfrak{m}}        \big|f(w_1\alpha_1,w_2\alpha_2,w_3 \alpha_3;X)\big|^{30}\dd \alpha_1  \dd \alpha_2\dd \alpha_3     \ll_{\mathbf{w},\ep}  X^{21-\frac{1}{8}+\ep}.
\]
\end{thm}

\subsection{Partial minor arcs over $\alpha_3$}\label{subsec-partial3}

 We consider the set $\mathfrak{m}_3$ of the $\alpha\in \big[\frac{1}{Q_0},1+\frac{1}{Q_0}\big]$ such that
whenever we have
$\ds{\Big|5\alpha-\frac{b_3}{q_3} \Big|\le \frac{1}{q_3^2}}$
with $(b_3;q_3)=1$, we have  $q_3>\tfrac{1}{5}X^{1/8}$.

\begin{lem} \label{lem-partial3}With the notation above, we have 
\[
\iiint_{[0,1]^2\times \mathfrak{m}_3}\big|f(\alpha_1,\alpha_2,\alpha_3;X)\big|^{30}  \dd \alpha_1  \dd \alpha_2\dd \alpha_3         \ll_{\ep} X^{21-\frac{1}{8}+\ep}
\]
\end{lem}

\begin{proof} First, we claim that
\[
\iiint_{[0,1]^2\times \mathfrak{m}_3}\big|f(\alpha_1,\alpha_2,\theta;X)\big|^{30}  \dd \alpha_1  \dd \alpha_2\dd \theta
\ll X^2  \iiint_{[0,1]^3\times \mathfrak{m}_3}\big|g_5(\boldsymbol{\alpha},\theta;X)\big|^{30}  \dd \boldsymbol{\alpha}\dd \theta
\]
where  $g_k(\boldsymbol{\alpha},\theta;X)$   is defined in \eqref{gk}.
Indeed, for fixed $\alpha_1,\alpha_2,\theta$, we consider the set $\Omega=\ZZ^{15}\cap[-X,X]^{15}$ and for $\omega=(x_1,x_2,\dots,x_{15})\in \Omega$, we consider $\Phi(\omega)$ such that ${\sum_ {\omega\in \Omega}\Phi(\omega)=\left( f(\alpha_1,\alpha_2,\theta;X)    \right)^{15}}$.
Now consider the set $\mathcal{H}$ of the integers of the form $\sum_{i=1}^{15}x_i^2-  \sum_{i=1}^{15} y_{i}^2$ with $|x_i|,|y_i|\le X$. Then we apply Lemma \ref{ampli} and use the bound  $\#\mathcal{H}\ll X^2$ : we have
\[
\big|f(\alpha_1,\alpha_2,\theta;X)\big|^{30}\ll X^2\int_{0}^1 \big| g_5(\boldsymbol{\alpha},\theta;X) \big|^{30}\dd\alpha_3,
\]
and the expected result follows by integrating over $\alpha_1,\alpha_2,\theta$.

\bigskip
We are now in a position to apply  Theorem  \ref{missing-deg} : we have
\[
 \iiint_{[0,1]^3\times \mathfrak{m}_3}\big|g_5(\boldsymbol{\alpha},\theta;X) \big|^{30}  \dd \boldsymbol{\alpha}\dd \theta
\ll   (\log(4X))^{30}   \Big(\sup_{\theta\in \mathfrak{m}_3}\psi_5^{*}(\theta;X) \Big)       J_{15,5}(4X+1)
\]
where $\psi_5^{*}(\theta;X)$ has been defined in \eqref{sup-psi}.
Now, for any $\mu\in \RR$ and any $\theta\in \RR$ there exists $b_3\in \ZZ$ and $q_3\ge 1$ coprime such that
$\ds{\Big|5\theta-\tfrac{b_3}{q_3} \Big|\le \tfrac{X^{1/8}}{q_3 X^5}}$, $q_3\le X^{5-\frac{1}{8}}$.
Then Lemma \ref{dist-entier} implies
\[
\frac{1}{X}\sum_{y=1}^X\min \Big(X^4,\tfrac{1}{ \|5\theta y +\mu\|}\Big)\ll X^3\big(1+\tfrac{X}{q_3}\big)+(1+\tfrac{q_3}{X})\log X
\]
and  since  $\theta\in  \mathfrak{m}_3$, this implies $\tfrac{1}{5}X^{1/8}< q_3\le X^{5-\frac{1}{8}}$, hence
\[
\sup_{\theta\in \mathfrak{m}_3}\psi_5^{*}(\theta;X)\ll X^{4-\frac{1}{8}+\ep}.
\]
We now conclude using the bound $J_{15,5}(4X+1)\ll X^{15+\ep}$ from \eqref{main-vino}.

\end{proof}

\subsection{ Partial minor arcs over $\alpha_2$}\label{subsec-partial2}

We now consider the set
\[
\mathfrak{W}_3=\bigcup_{q_3\le X^{1/8}}\bigcup_{b_3\in A_1(q_3)}\mathfrak{W}_3(q_3,b_3)
\]
where
\[
\mathfrak{W}_3(q_3,b_3):=\left[ \frac{b_3}{q_3}- \frac{X^{1/8}}{q_3X^5}, \frac{b_3}{q_3}+ \frac{X^{1/8}}{q_3X^5}  \right].
\]
By construction, one has
$\big[\frac{1}{Q_0},1+\frac{1}{Q_0}\big]\smallsetminus \mathfrak{m}_3\subset \mathfrak{W}_3$.
As previously, we consider the set $\mathfrak{m}_2$ of the $\alpha\in \big[\frac{1}{Q_0},1+\frac{1}{Q_0}\big]$ such that
whenever we have
$\ds{\big|\alpha-\tfrac{b_2}{q_2} \big|\le \tfrac{1}{q_2^2}}$
with $(b_2;q_2)=1$, this  implies $q_2>X^{3/4}$.

\begin{lem} \label{lem-partial2} With the notation above, we have
\[
\iiint_{[0,1]\times \mathfrak{m}_2\times \mathfrak{W}_3}\big|f(\boldsymbol{\alpha};X)\big|^{30}\dd\boldsymbol{\alpha}\ll_{\ep}X^{21-\frac{1}{8}+ \ep}
\]
\end{lem}
\begin{proof} Writing $\mathcal{U}_2=[0,1]\times \mathfrak{m}_2\times \mathfrak{W}_3$, we have
\[
\iiint_{ \mathcal{U}_2  }\big|f(\boldsymbol{\alpha};X)\big|^{30}\dd\boldsymbol{\alpha}
\ll \sup_{\boldsymbol{\alpha}\in  \mathcal{U}_2      }\big|f(\boldsymbol{\alpha};X)\big|^{10}\iiint_{[0,1]^2\times \mathfrak{W}_3}\big|f(\boldsymbol{\alpha};X)\big|^{20}\dd\boldsymbol{\alpha}
\]

Now, for any $\alpha_2\in \RR$, there exists $b_2\in \ZZ$ and $q_2\ge 1$ coprime such that
\[
\Big|\alpha_2-\tfrac{b_2}{q_2} \Big|\le \tfrac{X^{3/4}}{q_2X^3},\qquad q_2\le X^{3-\frac{3}{4}}.
\]
For such an $\alpha_2$, and for any $\alpha_1,\alpha_3\in \RR$,  Proposition \ref{vaughan-minor} applied to $\mathbf{k}=(1,3,5)$ and $b=10$ with \eqref{main-vino} implies
\[
f(\alpha_1,\alpha_2,\alpha_3;X)\ll_{\ep} X^{1+\ep} \left(\frac{q_2}{X^3}+\frac{1}{X}+\frac{1}{q_2}\right)^{1/20}.
\]
Since $\alpha_2\in \mathfrak{m}_2$, then  $q_2>X^{3/4}$, which implies
$\sup_{\boldsymbol{\alpha}\in  \mathcal{U}_2       }\big|f(\boldsymbol{\alpha};X)\big|^{10}\ll X^{10-\frac{3}{8}+\ep}$.
Moreover, 
\[
\iiint_{[0,1]^2\times \mathfrak{W}_3}\big|f(\boldsymbol{\alpha}     ;X)\big|^{20}\dd\boldsymbol{\alpha}
=\sum_{q_3\le \tfrac{1}{5}X^{1/8}}\sum_{b_3\in A_1(q_3)}\iiint_{[0,1]^2\times \mathfrak{W}_3(q_3,b_3)}\big|f(\boldsymbol{\alpha};X)\big|^{20}\dd\boldsymbol{\alpha}.
\]
Using Lemma \ref{alpha3-short}, for the inner integrals, we obtain
\[
\iiint_{[0,1]^2\times \mathfrak{W}_3}\big|f(\alpha_1,\alpha_2,\alpha_3;X)\big|^{20}\dd\boldsymbol{\alpha}\ll X^{11+  \frac{1}{4}+\ep}.
\]
Gathering these estimates gives the result announced.
\end{proof}

\subsection{Pruning}\label{subsec-pruning}

Let $\mathbf{w}\in \ZZ^n$ fixed with $w_1w_2\dots w_n\neq 0$. Using the notation of Lemma \ref{approx-dio}, we set
\be{}
\mathcal{L}=\bigcup_{1\le q\le M_0X^{7/8}}\bigcup_{\mathbf{a}\in A_3(q)}\mathcal{L}(q,\mathbf{a})
\ee
where  $\mathcal{L}(q,\mathbf{a})$ is the set of $(\alpha_1,\alpha_2,\alpha_3)\in [\frac{1}{Q_0}, 1+\frac{1}{Q_0}]^3$ such that
\[
\left|\alpha_1-\tfrac{a_1}{q}\right|\le \tfrac{1}{2|w_1|h(q,\mathbf{a},\mathbf{w})},\quad \left|\alpha_2-\tfrac{a_2}{q}\right|\le\tfrac{M_0X^{7/8}}{qX^3},\quad \left|\alpha_3-\tfrac{a_3}{q}\right|\le\tfrac{M_0X^{7/8}}{qX^5}.
\]

It is plain that $\mathfrak{M}\subset \mathcal{L}$. 

\begin{lem}\label{lem-pruning}One has
\[
\int_{\mathcal{L}\smallsetminus \mathfrak{M}}\big|f(w_1\alpha_1,w_2\alpha_2,w_3\alpha_3)\big|^{30}\dd\alpha_1\dd\alpha_2\dd\alpha_3\ll_{\mathbf{w},\ep} X^{21-2\tau+\ep}
\]
\end{lem}

\begin{proof} Throughout the proof we use the following notation
\[
\xi(\boldsymbol{\beta},X)=1+|\beta_1|X+|\beta_2|X^3+|\beta_3|X^5\qquad (\boldsymbol{\beta}\in \RR^3).
\]
The integral we have to estimate is equal to
\[
\sum_{q\le M_0X^{7/8}}\sum_{\mathbf{a}\in A_3(q)} \int_{\mathcal{L}(q,\mathbf{a})}\boldsymbol{1}_{\mathfrak{m}}(\alpha_1,\alpha_2,\alpha_3)\big|f(w_1\alpha_1,w_2\alpha_2,w_3\alpha_3)\big|^{30}\dd\alpha_1\dd\alpha_2\dd\alpha_3.
\]
Using Lemma \ref{approx-dio} and its notation, for $\boldsymbol{\alpha}\in \mathcal{L}(q,\mathbf{a})$, $\mathbf{w}\otimes \boldsymbol{\beta} $ satisfies \eqref{cond-beta} for some $\mathfrak{q}\asymp q$, hence  using \eqref{fk-bound2}, we have
\[
f(w_1\alpha_1,w_2\alpha_2,w_3\alpha_3)\ll Xq^{-\frac{1}{5}+\ep}   \xi(\boldsymbol{\beta},X)^{-1/5} +q^{\frac{4}{5}+\ep}
\]
uniformly for $q\le M_0X^{7/8}$, $\mathbf{a}\in A_3(q)$ and $\boldsymbol{\alpha}\in \mathcal{L}(q,\mathbf{a})$.
Thus, 
\[
\int_{\mathcal{L}\smallsetminus \mathfrak{M}}\big|f(w_1\alpha_1,w_2\alpha_2,w_3\alpha_3)\big|^{30}\dd\alpha_1\dd\alpha_2\dd\alpha_3\ll X^{\ep}S_1 + X^{17+\ep}
\]
where we have set
\[
S_1:= \sum_{q\le M_0X^{7/8}}X^{30}q^{-6}\sum_{  \mathbf{a}\in A_3(q) }\int_{ \mathcal{L}(q,\mathbf{0})} \boldsymbol{1}_{\mathfrak{m}}\Big(\tfrac{\mathbf{a}}{q}+\boldsymbol{\beta}\Big)     \xi(\boldsymbol{\beta},X)^{-6}\dd\beta_1\dd\beta_2\dd\beta_3.
\]
Since 
\[
\int_{ \mathcal{L}(q,\mathbf{0})} \boldsymbol{1}_{\mathfrak{m}}\Big(\frac{\mathbf{a}}{q}+\boldsymbol{\beta}\Big)  \xi(\boldsymbol{\beta},X)^{-6}\dd\beta_1\dd\beta_2\dd\beta_3
\ll X^{-9}\int_{\RR^3}\xi(\boldsymbol{\beta},1)^{-6}\dd\beta_1\dd\beta_2\dd\beta_3\ll X^{-9},
\]
the contribution of the $q>\tfrac{1}{2}X^{\tau}$ in $S_1$  is $\ll X^{21-2\tau}$,  which does not exceed the expected bound.
Now, since for $\frac{\mathbf{a}}{q}+\boldsymbol{\beta}\in \mathfrak{m}$, we have 
\[
q+q|\beta_1|X+q|\beta_2|X^3+q|\beta_3|X^5>X^{\tau},
\]
Hence, for $q\le \tfrac{1}{2}X^{\tau}$, we have  $|\beta_1|X+|\beta_2|X^3+|\beta_3|X^5>\frac{X^{\tau}}{2q}$, which implies
\[
\int_{ \mathcal{L}(q,\mathbf{0})} \boldsymbol{1}_{\mathfrak{m}}\Big(\frac{\mathbf{a}}{q}+\boldsymbol{\beta}\Big)   \xi(\boldsymbol{\beta},X) ^{-6}\dd\boldsymbol{\beta}
\ll X^{-9}\int_{H_3(\frac{X^{\tau}}{2q})} \xi(\boldsymbol{\beta},1)^{-6}\dd\boldsymbol{\beta}          \ll X^{-9}\Big(1+\frac{X^{\tau}}{q}\Big)^{-3}
\]
by using (iii) of  Lemma \ref{maj-int-2}. Thus, the contribution of the $q\le \tfrac{1}{2}X^{\tau}$ in $S_1$ is
\[
\ll \sum_{q\le  \tfrac{1}{2}X^{\tau}}X^{21}q^{-3}\Big(1+\frac{X^{\tau}}{q}\Big)^{-3}\ll X^{21-2\tau}.
\]

\end{proof}

\subsection{Proof of Theorem \ref{th-minor2}}

We keep the notation $\mathfrak{m}_3$, $\mathfrak{m}_2$, $\mathfrak{W}_3$ and $\mathcal{L}$ introduced in the sections \ref{subsec-partial3},
 \ref{subsec-partial2} and \ref{subsec-pruning}. Similarly to $\mathfrak{W}_3$, we construct the set
\[
\mathfrak{W}_2=\bigcup_{q_2\le X^{3/4}}\bigcup_{b_2\in A_1(q_2)}\mathfrak{W}_2(q_2,b_2)
\]
where
\[
\mathfrak{W}_2(q_2,b_2):=\left[ \frac{b_2}{q_2}- \frac{X^{3/4}}{q_2X^3}, \frac{b_2}{q_2}+ \frac{X^{3/4}}{q_2X^3}  \right].
\]
Again by construction, one has
$\big[\frac{1}{Q_0},1+\frac{1}{Q_0}\big]\smallsetminus \mathfrak{m}_2\subset \mathfrak{W}_2$.

For a fixed $\mathbf{w}\in \ZZ^3$ such that $w_1w_2w_3\neq 0$, we introduce various subsets of $\mathfrak{m}$ :

\begin{itemize}
\item We define the set $\mathfrak{n}_1$  of  $\boldsymbol{\alpha}\in \mathfrak{m}$ such that $\mathbf{w}\otimes \boldsymbol{\alpha}\mod 1$ belongs  $ [0,1]\times \mathfrak{W}_2\times \mathfrak{W}_3$.
\item We also define  $\mathfrak{n}_2$,  the set of $\boldsymbol{\alpha}\in \mathfrak{m}$ such that $\mathbf{w}\otimes \boldsymbol{\alpha}\mod 1$ belongs to $[0,1]\times \mathfrak{m}_2\times \mathfrak{W}_3$.
\item Finally $\mathfrak{n}_3$ is  the set of $\boldsymbol{\alpha}\in \mathfrak{m}$ such that $\mathbf{w}\otimes \boldsymbol{\alpha}\mod 1$ belongs to $[0,1]^2\times \mathfrak{m}_3$.
\end{itemize}
It is plain that  $\mathfrak{m}=\mathfrak{n}_1\cup \mathfrak{n}_2 \cup\mathfrak{n}_3$. Writing
\[
I_j:=\iiint_{\mathfrak{n}_j}\big|f(w_1\alpha_1,w_2\alpha_2,w_3\alpha_3;X)\big|^{30}  \dd \alpha_1  \dd \alpha_2\dd \alpha_3\qquad (1\le j\le 3),
\]
we now have 
\[
\iiint_{\mathfrak{m}}\big|f(w_1\alpha_1,w_2\alpha_2,w_3\alpha_3;X)\big|^{30}  \dd \alpha_1  \dd \alpha_2\dd \alpha_3=I_1+I_2+I_3.
\]
Using Lemma \ref{lem-partial3}, we deduce
\[
I_3 \ll \iiint_{[0,1]^2\times \mathfrak{m}_3}\big|f(\alpha_1,\alpha_2,\alpha_3;X)\big|^{30}  \dd \alpha_1  \dd \alpha_2\dd \alpha_3\ll X^{21-\frac{1}{8}+\ep}.
\]
Similarly
\[
I_2 \ll \iiint_{[0,1]\times \mathfrak{m}_2\times \mathfrak{W}_3}\big|f(\alpha_1,\alpha_2,\alpha_3;X)\big|^{30}  \dd \alpha_1  \dd \alpha_2\dd \alpha_3\ll X^{21-\frac{1}{8}+\ep}
\]
by using Lemma \ref{lem-partial2}.

 Finally,  we have $\mathfrak{n}_1\subset \mathcal{L}\smallsetminus \mathfrak{M}$. Indeed,  any $\boldsymbol{\alpha}\in \mathfrak{n}_1$ satisfies
 $\displaystyle{\left|w_2\alpha_2-\tfrac{b_2}{q_2}\right|\le \tfrac{X^{3/4}}{q_2X^{3}}}$ and      $\displaystyle{ \left|w_3\alpha_3-\tfrac{b_3}{q_3}\right|\le \tfrac{X^{1/8}}{q_3X^{5}}
}$
 for some $b_i,q_i$ with $q_2\le X^{3/4}$ and $q_3\le X^{1/8}$. Therefore,  using Lemma  \ref{approx-dio}, we have $\boldsymbol{\alpha}\in  \mathcal{L}\smallsetminus \mathfrak{M}$.  
 
 We may now use Lemma \ref{lem-pruning} : this implies $I_1 \ll X^{21-2\tau+\ep}$, which completes the proof.

\bigskip
\textsc{Simon Boyer, Universit\'e de Lyon, Universit\'e Claude-Bernard Lyon 1,   CNRS UMR 5208, Institut Camille Jordan,   F-69622 Villeurbanne, France.} 
\par
\textit{E-mail address} : \texttt{simonboyer7@gmail.com}

\bigskip
\textsc{Olivier Robert, Universit\'e de Lyon,  Universit\'e de Saint-\'Etienne,   CNRS UMR 5208, Institut Camille Jordan,   F-42000 Saint-\'Etienne, France.}
\par
\textit{E-mail address} :  \texttt{olivier.robert@univ-st-etienne.fr}

\end{document}